\documentclass{siamart190516}

\usepackage{float}
\usepackage{savesym}
\usepackage{graphicx}
\usepackage{scrtime}
\usepackage[dont-mess-around]{fnpct}
\usepackage[misc]{ifsym}
\usepackage{etoolbox}
\usepackage{multirow}
\usepackage{amsmath}
\usepackage{amssymb}
\usepackage{tabularx}
\usepackage{cleveref}

\usepackage{caption}
\usepackage{subcaption}
\usepackage{wrapfig}
\usepackage{cases}
\usepackage{mathtools}
\usepackage{booktabs}
\usepackage{soul,color}

\DeclareMathOperator*{\argmin}{arg\,min}

\title{The Deep Minimizing Movement Scheme}

\author{
Min Sue Park\footnotemark[2]\,\,\footnotemark[1]
\and Cheolhyeong Kim\footnotemark[2]\,\,\footnotemark[1]
\and Hwijae Son\footnotemark[3]\,\,\footnotemark[1]
\and Hyung Ju Hwang\footnotemark[2]\,\,\footnotemark[4]}

\headers{The Deep Minimizing Movement Scheme}{M. S. Park, C. Kim, H. Son, H. J. Hwang}

\begin{document}
\maketitle

\renewcommand{\thefootnote}{\fnsymbol{footnote}}
\footnotetext[2]{Department of Mathematics, Pohang University of Science and Technology, Pohang 790-784, Republic of Korea (\email{minsuepark@postech.ac.kr},\email{tyty4@postech.ac.kr},\email{hjhwang@postech.ac.kr}).}
\footnotetext[3]{Stochastic Analysis and Application Research Center, Korea Advanced Institute of Science and Technology, Daejeon, Republic of Korea (\email{son9409@kaist.ac.kr}).}
\footnotetext[1]{Joint first authors}
\footnotetext[4]{Corresponding Author}
\renewcommand{\thefootnote}{\arabic{footnote}}

\begin{abstract}
    Solutions of certain partial differential equations (PDEs) are often represented by the steepest descent curves of corresponding functionals. Minimizing movement scheme was developed in order to study such curves in metric spaces. Especially, Jordan-Kinderlehrer-Otto studied the Fokker-Planck equation in this way with respect to the Wasserstein metric space. In this paper, we propose a deep learning-based minimizing movement scheme for approximating the solutions of PDEs. The proposed method is highly scalable for high-dimensional problems as it is free of mesh generation. We demonstrate through various kinds of numerical examples that the proposed method accurately approximates the solutions of PDEs by finding the steepest descent direction of a functional even in high dimensions.
\end{abstract}



\section{Introduction}\label{intro}
Starting from the pioneering concerns given in \cite{de1980problems}, the theory of gradient flow evolution in general metric spaces has been intensively studied to unify many problems in calculus of variations, differential equations, mean curvature flow, etc. In this trend, \textit{minimizing movement scheme} \cite{de1992movimenti} was devised as a general method for the study of steepest descent curves of a functional $F$ in a metric space $(X,d_X)$, and this method was extensively applied to various differential equations as a time-implicit semi-discrete scheme.

Historically, the application of minimizing movement scheme to differential equations was exclusively based on gradient flows arising from a separable Hilbert space, particularly $L^2(\mathbb{R}^n)$. Later, in \cite{jordan1998variational}, Jordan, Kinderlehrer, and Otto observed that the Fokker-Planck equation can be understood as a steepest descent for the free energy with respect to the Wasserstein distance and utilized the minimizing movement scheme for Wasserstein spaces, and that is why the minimizing movement scheme is now widely known as the JKO-scheme (Jordan, Kinderlehrer and Otto). In recent years, it has been known that a large class of diffusion equations can be variationally reformulated as gradient flows in $L^2$-Wasserstein spaces, and the minimizing movement scheme has now become a popular tool for deriving numerical schemes for these equations. In addition to the minimizing movement scheme, another branch of numerical scheme for gradient flow called Scalar Auxiliary Variable (SAV) scheme has been recently proposed and developed \cite{shen2018scalar, zhang2022generalized, jiang2022improving}.

Neural networks have been considered as candidates for approximators of the solutions of PDEs in past decades. The earliest work dates back to the 1990s, for example, \cite{lagaris1998artificial,dissanayake1994neural}. However, due to the limitation of computing resources and the absence of an efficient training algorithm, the proposed neural network solvers were often undervalued at that time. Along with the recent advances in deep learning theory such as the back-propagation and the stochastic gradient descent algorithm, several approaches for solving PDEs via neural networks have been revisited. One major line of research is Physics Informed Neural Networks (PINNs) \cite{raissi2019physics}, which proposed to train a neural network to minimize the sum of PDE residual functions. Plenty of works are then reported regarding the convergence property of continuous loss formulation \cite{NHM,hwang2020trend,lee2020model}, dealing with the exact imposition of the boundary condition \cite{berg2018unified,muller2021notes}, works on improving training efficiency \cite{son2021sobolev,wang2020and,mcclenny2020self,van2020optimally}, and applications to high-dimensional problems \cite{sirignano2018dgm} to name a few. Another approach, called the deep Ritz method, reformulates a PDE problem into an optimization problem of a variational functional \cite{weinan2018deep}, and a corresponding convergence result is reported in \cite{muller2019deep}. Neural network solutions of variational problems with essential boundary conditions are studied in \cite{liao2019deep}, and a pre-training strategy for the deep Ritz method condition is proposed in \cite{courte2021robin}.

Traditional numerical schemes such as Finite Difference Method (FDM), Finite Element Method (FEM), Finite Volume Method (FVM) require fine mesh generation for an accurate approximation. Since mesh generation and computation become intractable as the computational dimension increases, they often fail to solve PDEs in high dimensions \cite{lu2021deepxde}. On the other hand, the aforementioned neural network methods have successful approximations in high-dimensional settings. For example, \cite{sirignano2018dgm} proposed an iterative sampling technique together with the stochastic gradient descent to tackle the high-dimensional PDEs, and \cite{sirignano2018dgm,son2021sobolev} showed that neural networks could accurately approximate the solution in high dimensions. Additionally, \cite{weinan2018deep} considers variational problems in high dimensions for elliptic PDEs. In this paper, we consider several high-dimensional PDEs as example problems to demonstrate the empirical success of our method.

Optimizing a functional and finding a minimizer have been widely studied both in the machine learning community and in the PDE literature. Recently, optimization of a functional defined on a space of probability measures equipped with $L^2$-Wasserstein distance arises as an important problem. Therefore, Wasserstein gradient flow plays an important role in studying the steepest descent of a target functional in the space of probability measures. 
\cite{alvarez2021optimizing} and \cite{mokrov2021large} proposed to applying JKO scheme with neural networks with tractable algorithms. However, their usage of JKO scheme mainly focused on a fast approximation of the steepest descent of the target functional, not on the accurate prediction of a whole dynamic system. For example, the experiment results of \cite{alvarez2021optimizing} on the Wasserstein gradient flow type PDEs were mainly discussed on its \textit{asymptotic convergence of the objective functions}, but the accurate prediction of the solution $u(t,x)$ was not mentioned.
Meanwhile, \cite{liu2020neural} proposed to solve Fokker-Planck equation with neural networks by emerging JKO scheme, Wasserstein Gradient Flow, and its geometric interpretation. However, its usage is limited to Fokker-Planck equation, since \cite{liu2020neural} projected the Fokker-Planck equation onto a parametric manifold with a new metric tensor and viewed the equation as a new ODE system on the manifold. 

In this paper, we propose a neural network model which reflects the minimizing movement scheme for solving PDEs. Our proposed method has several advantages in the following aspects. Firstly, our method is more suitable for solving high-dimensional problems than the traditional mesh-based schemes \cite{liu2020neural} because the  mesh generation is not necessary for the proposed method, and neural networks are capable of handling high-dimensional objects in general. Secondly, while specifying boundary condition often causes a problem in both residual minimization \cite{muller2021notes} and the deep Ritz method \cite{courte2021robin}, it often follows naturally from the minimizing movement scheme in our method. Therefore, we do not need to explicitly specify boundary conditions for those PDEs derivable from the minimizing movement scheme. Lastly, our model easily enables time extrapolation of the numerical solution in contrast to residual minimization; we only need to solve one more optimization problem for one-step time extrapolation, while the residual minimization requires additional training in the whole time-space domain. Furthermore, an increment of the error at each time step is bounded by $\mathcal{O}(\tau^{1/2})$, where $\tau$ denotes the step size. See, theorem \ref{theorem4}.




\section{Preliminaries}\label{preliminary}

\subsection{Gradient flow}
In general, given a Riemannian manifold $(M,g)$, a curve $u:(0,\infty) \rightarrow M$ is said to be the \textit{gradient flow} of a smooth functional $\mathcal{F}:M \rightarrow \mathbb{R}$ if it follows the direction in which $\mathcal{F}$ decreases at most. Formally, this definition can be interpreted by the evolution equation 
\begin{equation}\label{gf_def}
\frac{d}{dt} u(t) = -\nabla_M \mathcal{F}(u(t)).
\end{equation}

If $M$ is a Hilbert space, the gradient of $\mathcal{F}$ is given by its functional derivative, which is derived from the Riesz representation of the Fréchet derivative $D\mathcal{F}$ of $\mathcal{F}$: $D\mathcal{F}(x)(y) = \langle \nabla_M \mathcal{F}(x), y \rangle$ for every $x,y \in M$. Therefore, for $\Omega \subset \mathbb{R}^n$, $\nabla_{L^2(\Omega)}\mathcal{F}(u)$ is merely the functional derivative $\frac{\delta \mathcal{F}}{\delta u}(u)$ since $L^2(\Omega)$ is itself a Hilbert space.

For example, the heat equation $\frac{d}{dt}u(t) = \Delta u$ is the gradient flow of the Dirichlet energy $\mathcal{F}(u):=\frac{1}{2}\int ||\nabla u||^2 dx$ in $L^2(\mathbb{R}^n)$ because  $\nabla_{L^2(\mathbb{R}^n)}\mathcal{F}(u) = \frac{\delta \mathcal{F}}{\delta u}(u) = -\Delta u$.

As illustrated in the above example, to reformulate an evolution equation into its variational form as in (\ref{gf_def}), finding an appropriate Riemannian manifold $(M,g)$ would be helpful. An important and interesting example of such Riemannian manifold was observed in \cite{otto2001geometry} that there is a metric tensor $g$ which induces the $L^2$-Wasserstein metric $\mathbb{W}_2$ on $\mathcal{P}_2(\Omega)$, the space of Borel probability distributions on $\Omega \subset \mathbb{R}^n$ with finite second moments. Here, $\Omega$ is assumed to be either an open bounded subset of $\mathbb{R}^n$ with smooth boundary or $\mathbb{R}^n$. Abusing notation and identifying the absolutely continuous measure $u$ with its density as $du(x) = u(x)dx$, it was shown in \cite{otto2001geometry} that, under some assumptions on $\mathcal{F}$, the gradient $\nabla_{\mathbb{W}_2}\mathcal{F}(u)$ can be explicitly expressed by
\begin{equation}
\nabla_{\mathbb{W}_2}\mathcal{F}(u) = -\nabla \cdot (u \nabla \frac{\delta \mathcal{F}}{\delta u}(u)).
\end{equation}

In general, many interesting linear functionals $\mathcal{F}$ on $(\mathcal{P}_2(\mathbb{R}^n), \mathbb{W}_2)$ are given as linear combinations of three types of basic functionals; potential energy $\mathcal{V}(u):=\int V u $, interaction energy $\mathcal{W}(u):= \int (W\ast u)u$, and internal energy  $\mathcal{E}(u):=\int u\phi(u)$, where these energies are defined for absolutely continuous $u$'s and defined by $\infty$ otherwise. For example, the heat equation $\frac{d}{dt}u(t) = \nabla u$ is also the gradient flow of the entropy $\mathcal{E}:(\mathcal{P}_2(\mathbb{R}^n),\mathbb{W}_2) \rightarrow [0,+\infty]$ defined by $\int u(x) \log u(x) dx$ for absolutely continuous $u$ and $\infty$ otherwise, because $\nabla \cdot (u\nabla \frac{\delta \mathcal{E}}{\delta u}(u)) = \Delta u$.

Examples of basic functionals and their corresponding gradient flows on $L^2(\mathbb{R}^n)$ and $(\mathcal{P}_2(\Omega),\mathbb{W}_2)$ are given in Table \ref{Tab:ef}.

\begin{table}[]
\centering
\begin{tabularx}{0.9\textwidth}{c|X|X|}
\cline{2-3}
\textbf{}                                 & \multicolumn{1}{c|}{\textbf{Functional}} & \multicolumn{1}{c|}{\textbf{Gradient Flow}} \\ \hline
\multicolumn{1}{|c|}{\multirow{3}{*}{$L^2(\mathbb{R}^n)$}} & $\mathcal{F}(u)=\int \phi(u)$                  & $\frac{d}{dt}u(t)= - \phi'(u)$                     \\
\multicolumn{1}{|c|}{}                    & $\mathcal{F}(u)=\int ||\nabla u||^2$                  & $\frac{d}{dt}u(t) = \Delta u$                      \\
\multicolumn{1}{|c|}{}                    & $\mathcal{F}(u)=\int u(x)f(x)$                   & $\frac{d}{dt}u(t)=-f(x)$                      \\ \hline
\multicolumn{1}{|c|}{\multirow{4}{*}{$\mathbb{W}_2$}} & $\mathcal{F}(u)=\int u \log u$                  & $\frac{d}{dt}u(t) = \Delta u$                     \\
\multicolumn{1}{|c|}{}                    & $\mathcal{F}(u)=\frac{1}{m-1}\int u^m$                   & $\frac{d}{dt}u(t) = \Delta u^m$                      \\
\multicolumn{1}{|c|}{}                    & $\mathcal{F}(u)=\int Vu$                   & $\frac{d}{dt}u(t)=\nabla \cdot (\nabla V u)$                      \\
\multicolumn{1}{|c|}{}                    & $\mathcal{F}(u)=\int (W\ast u)u$                   & $\frac{d}{dt}u(t)=\nabla \cdot(\nabla(W\ast u) u)$                      \\ \hline
\end{tabularx}
\caption{Examples of functionals and their corresponding gradient flows.}
\label{Tab:ef}
\end{table}

\subsection{Minimizing movement scheme (JKO scheme)}
A metric space $(X,d_X)$ is said to be a geodesic space if $d_X(x,y)$ is a geodesic between $x$ and $y\in X$ for every $x$ and $y\in X$. It is well-known that if a metric space $X$ is a geodesic Polish space, then $(\mathcal{P}_2(X), \mathbb{W}_2)$ is a geodesic Polish space as well, and any Hilbert space is a geodesic space \cite{ambrosio:hal-00769391}. Therefore, for any convex subset $\Omega \subset \mathbb{R}^n$, $L^2(\Omega)$ and $(\mathcal{P}_2(\Omega), \mathbb{W}_2)$ are geodesic Polish spaces.

Let $(X,d_X)$ be a geodesic Polish space with Riemannian structure and $F:X\rightarrow \mathbb{R}\cup \{+\infty\}$ be a lower semi-continuous $\lambda$-geodesically convex function (under suitable compactness assumptions to guarantee the existence of a minimum), and iteratively define
\begin{equation}
    x_{k+1}^\tau = \argmin_{x\in X} \{\mathcal{F}(x) + \frac{d_X(x,x_k^\tau)^2}{2\tau}\},  \quad x_0^\tau = x_0,
\end{equation}
for $\tau >0$ and some initial point $x_0 \in X$ such that $E(x_0)<+\infty$.

This scheme is said to be the \textit{minimizing movement scheme}. If we define a curve such that $x^\tau(k\tau) = x_k^\tau$ and such that $x^\tau$ restricted on any interval $[k\tau, (k+1)\tau]$ is a constant-speed geodesic with speed $\frac{d(x_k^\tau, x_{k+1}^\tau)}{\tau}$, up to a subsequence $\tau_j\to 0$, then $x^\tau$ converges uniformly to the solution of (\ref{gf_def}) when some regularity conditions on $\mathcal{F}$ and $(X,d_X)$ are assumed.

Moreover, a gradient flow derived through the minimizing movement scheme in $L^2(\Omega)$ is always accompanied by the boundary condition $\frac{\delta \phi}{\delta (\nabla u)} \cdot \hat{\boldsymbol{n}}=0$ on $\partial \Omega$ when the functional $\mathcal{F}$ is given in the form $\mathcal{F}(u):=\int \phi(u,\nabla u)$, if $\Omega$ is not $\mathbb{R}^n$ itself \cite{braides2014local}. Likewise, a gradient flow derived through the minimizing movement scheme in $(\mathcal{P}_2(\Omega), \mathbb{W}_2)$ is always accompanied by no-flux boundary condition on $\partial \Omega$, i.e. $u \nabla \frac{\delta \mathcal{F}}{\delta u}(u) \cdot \hat{\boldsymbol{n}} = 0$, if $\Omega$ is not $\mathbb{R}^n$ itself \cite{santambrogio2015optimal}.

In particular, the gradient flows of the heat equation and the Allen-Cahn equation derived through the minimizing movement scheme in the space $L^2(\Omega)$ satisfy the Neumann boundary condition \cite{braides2014local,mizuno2015convergence}. Similarly, the gradient flow of the heat equation derived through the minimizing movement scheme in the space $(\mathcal{P}_2(\Omega), \mathbb{W}_2)$ satisfies the Neumann boundary condition \cite{ambrosio:hal-00769391}. In this paper, we consider Neumann boundary conditions when $\Omega \neq \mathbb{R}^n$.

\subsection{Optimal transport and \texorpdfstring{$\boldsymbol{L^2}$}{TEXT}%
-Wasserstein distance}\label{optimaltransport}

Let $\mu,\nu$ be Borel probability measures on Polish spaces $X, Y$ respectively. Given a cost function $c:X\times Y\rightarrow [0,+\infty]$, where $c(x,y)$ measures the cost of transporting one unit of mass from $x\in X$ to $y\in Y$, \textit{the optimal transport problem} is how to tranposrt $\mu$ to $\nu$ while minimizing the cost $c$. This was firstly formulated by Monge \cite{monge1781memoire} as follows:
\begin{equation}\label{Monge}
   (MP) \quad \inf_{T_\ast \mu = \nu} \int_X c(x,T(x))d \mu(x),
\end{equation}
where $T:X\rightarrow Y$ is a measurable map (transport map).

However, this problem is ill-posed because of the constraint $T_\ast \mu = \nu$. For instance, if $\mu$ is a Dirac measure and $\nu$ is not, there is no admissible $T$. Later, Kantorovich proposed the following way to relax the Monge problem:
\begin{equation}\label{Kantorovich}
    (KP) \quad \min_{\gamma \in \Pi(\mu, \nu)} \int_{X\times Y} c(x,y) d\gamma(x,y),
\end{equation}
where $\Pi(\mu,\nu)$ is the collection of joint distributions (transport plans) of $\mu$ and $\nu$.

Moreover, if $\mu$ is non-atomic and $c$ is continuous, then there always exist optimal transport plans satisfying (\ref{Kantorovich}) and it coincides with the infimum of (\ref{Monge}) \cite{pratelli2007equality}. In particular, this is the case when $X=Y=\mathbb{R}^n$ and $c(x,y):=||x-y||_2^2$. The $L^2$\textit{-Wasserstein distance} between $\mu$ and $\nu$ is defined to be the square root of the minimum of (\ref{Kantorovich}) when $c(x,y):=||x-y||_2^2$, and is denoted by $\mathbb{W}_2(\mu,\nu)$.


Estimation of $L^2$-Wasserstein distance is essential for solving PDEs with JKO scheme and applying optimal transport in a machine learning area. Previous studies \cite{taghvaei20192,makkuva2020optimal} proposed to estimate this $L^2$-Wasserstein distance by applying Kantorovich-Rubinstein duality theorem \cite{villani2003topics} and approximating primal/conjugate convex functions. However, their computations were based on minimax optimization, which leads to a heavy computational cost and a slow estimation speed.

Meanwhile, the authors in \cite{korotin2019wasserstein} proposed Wasserstein-2 generative networks (W2GN), an end-to-end non-minimax algorithm for training optimal transport mappings for $L^2$-Wasserstein distance. They avoided the minimax optimization by introducing a regularization term on the convex conjugate potential. Instead of optimizing minimax problem, W2GN simply minimizes the cost and controls the conjugate function. Compared with the previous minimax-based computation algorithms in \cite{taghvaei20192,makkuva2020optimal}, W2GN significantly improves the computational cost for estimating $L^2$-Wasserstein distance.

In the rest of this subsection, we briefly describe the method for estimating $\mathbb{W}_2(\mu,\nu)$, that was proposed in \mbox{\cite{korotin2019wasserstein}}.

When absolutely continuous distributions $u,\rho$ with finite second moments are given, by re-arranging the Monge's formulation (\ref{Monge}), $\mathbb{W}_2(u,\rho)$ can be expressed by 
\begin{equation}\label{middle_step}
    \mathbb{W}_2^2(u,\rho)=\int_{\mathbb{R}^n} ||x||^2 du(x) + \int_{\mathbb{R}^n} ||y||^2 d\rho(y)- 2\cdot \max_{T_\ast u = \rho} \int_{\mathbb{R}^n} \langle x, T(x)\rangle du(x).
\end{equation}
We denote the maximum term on the right hand side of (\ref{middle_step}) by $\text{Corr}(u,\rho)$, and $-\text{Corr}(u,\rho)$ can be regarded as (\ref{Monge}) for $c(x,y):=-\langle x, y \rangle$ \cite{mccann1995existence}. Therefore, by applying the Kantorovich-Rubinstein duality theorem, we have
\begin{equation}\label{corr}
    \text{Corr}(u,\rho) = \min_{\psi\in \text{Convex}}\bigg[\int_{\mathbb{R}^n} \psi(x) du(x) + \int_{\mathbb{R}^n} \overline{\psi}(y)d\rho(y)\bigg],
\end{equation}
where the minimum is taken over all the convex functions $\psi:\mathbb{R}^n \rightarrow \mathbb{R}\cup\{+\infty\}$, and $\overline{\psi}(y):=\sup_{x\in \mathbb{R}^n}(\langle x, y \rangle - \psi(x))$ is the Fenchel convex conjugate to $\psi$.

According to \cite{villani2003topics,mccann1995existence}, the gradient $T^*:=\nabla \psi^*$ of the optimal $\psi^*$ readily gives the maximizer of (\ref{middle_step}), and we have
\begin{equation}\label{relation}
    (T^*)^{-1}(y)=(\nabla \psi^*)^{-1}(y) = \nabla \overline{\psi^*}(y).
\end{equation}
Therefore, it suffices to find the minimizer $\psi$ of (\ref{corr}) to evaluate the $L^2$-Wasserstein distance $\mathbb{W}_2(u,\rho)$. However, to do this, we need to solve the optimization sub-problem $\overline{\psi}(y)=\max_{x\in \mathbb{R}^n} (\langle x, y \rangle - \psi(x))$, and this would lead to a minimax problem
\begin{align}\label{minimaxeq}
    &\min_{\psi\in\text{Convex}}\bigg[ \int_{\mathbb{R}^n} \psi(x) du(x) + \int_{\mathbb{R}^n} \bar{\psi}(y) d\rho(y) \bigg] \nonumber \\
    &= \min_{\psi\in\text{Convex}}\bigg[ \int_{\mathbb{R}^n} \psi(x) du(x) + \int_{\mathbb{R}^n} [\max_{x\in\mathbb{R}^n} \langle x , y \rangle - \psi(x)] d\rho(y) \bigg] \nonumber \\
    &= \min_{\psi\in\text{Convex}}\bigg[ \int_{\mathbb{R}^n} \psi(x) du(x) + \max_{S:\mathbb{R}^n\rightarrow\mathbb{R}^n}\int_{\mathbb{R}^n} [\langle S(y) , y \rangle - \psi(S(y))] d\rho(y) \bigg]
\end{align}
where the maximum is taken over arbitrary measurable functions $S$.

Nevertheless, \cite{korotin2019wasserstein} proposed a trick to convert this problem into an end-to-end non-minimax problem by considering the optimal $\psi$ and $\bar{\psi}$. For the variational approximation in (\ref{minimaxeq}), we obtain a variational lower bound which matches the entire value for $S=(\nabla \psi)^{-1}(y) = \nabla \overline{\psi}(y)$. 

Therefore, the primal potential $\psi$ and its conjugate $\overline{\psi}$ can be approximated by two parametrized convex functions $\psi_\theta$ and $\overline{\psi_\omega}$ respectively, if we minimize the following objective:
\begin{equation}\label{middle}
    \min_{\theta, \omega}\bigg[ \int_{\mathbb{R}^n} \psi_\theta(x) du(x) + \int_{\mathbb{R}^n} [\langle \nabla \overline{\psi_\omega}(y),y\rangle - \psi_\theta(\nabla \overline{\psi_\omega}(y))] d\rho(y) \bigg].
\end{equation}
Moreover, the relation (\ref{relation}) must be imposed on the optimized pair $(\psi_\theta,\overline{\psi_\omega})$, and this can be done by additionally minimizing the following regularization:
\begin{equation}
    R(\theta,\omega):= \int_{\mathbb{R}^n} ||\nabla \psi_\theta \circ \nabla \overline{\psi_\omega}(y) - y ||^2 d\rho(y).
\end{equation}
In conclusion, our final objective would be the following:
\begin{equation}\label{W2objective}
    \min_{\theta, \omega}\bigg[ \int_{\mathbb{R}^n} \psi_\theta(x) du(x) + \int_{\mathbb{R}^n} [\langle \nabla \overline{\psi_\omega}(y),y\rangle - \psi_\theta(\nabla \overline{\psi_\omega}(y))] d\rho(y) + \lambda\cdot R(\theta,\omega) \bigg],
\end{equation}
where $\lambda>0$ is a hyperparameter.

Using the identities (\ref{middle_step}) and (\ref{corr}), once the optimal $\psi_{\theta^*}, \overline{\psi_{\omega^*}}$ are found, $\mathbb{W}_2(u,\rho)$ can now be approximated by the following:
\begin{equation}\label{W2approx}
    \int_{\mathbb{R}^n} ||x||^2 du(x) + \int_{\mathbb{R}^n} ||y||^2 d\rho(y)- 2\cdot \bigg[\int_{\mathbb{R}^n} \psi_{\theta^*}(x) du(x) + \int_{\mathbb{R}^n} \overline{\psi_{\omega^*}}(y)d\rho(y)\bigg]
\end{equation}

The remaining obstacle is to find the optimal $\psi_{\theta^*}$ and $\overline{\psi_{\omega^*}}$ with a tractable computation. In order to find them and to use them as approximated solutions of $\psi^*$ and $\overline{\psi^*}$, a large enough but tractable space of parametrized convex functions should be introduced. Here, we used input convex neural network (ICNN) architectures to approximate the optimal convex functions, which were firstly presented in \cite{amos2017input}. Recently, ICNN architectures are widely chosen to approximate the optimal convex potential and its conjugate in the previous studies \cite{taghvaei20192,makkuva2020optimal,korotin2019wasserstein,alvarez2021optimizing,mokrov2021large}. Among them, a fully-convex, $k$-layer, fully-connected ICNN (FICNN) architecture was proposed
\begin{align*}
    z_{i+1} = g_i(W_i^{(z)}z_i + W_i^{(y)}y + b_i),\quad 0\leq i \leq k-1,\quad  W_0^{(z)} = 0,
\end{align*}
where $y$ denotes an input, $g_i$ denotes the activation function, and $z_k$ denotes the final output. Let $\theta := \{W_{0:k-1}^{(y)},W_{1:k-1}^{(z)},b_{0:k-1}\}$ be the parameters of $k$-layer FICNN $f(y;\theta)=z_k$. The authors in \cite{amos2017input} claimed that the function $f$ is convex with respect to $y$ given $W_{1:k-1}^{(z)}$ are non-negative and $g_i$ are non-decreasing convex activation functions. Following the method of \cite{korotin2019wasserstein}, the non-negativity of the necessary parameters in our experiments with FICNN can be attained by clipping each parameter to be non-negative for each gradient descent update.

\subsection{Universal approximation theorem}
The approximation theory of the neural network has been widely studied in past decades after a seminal work \cite{cybenko1989approximation} which states the universal approximation property of one hidden layer neural network with sigmoidal activations.

The original result had been generalized in many directions. For example, \cite{hornik1991approximation} showed the approximation property for more general activation functions.
\begin{theorem}[Theorem 2 in \cite{hornik1991approximation}]\label{theorem2}
    Whenever $\sigma$ is continuous, bounded and non-constant, the finite sums of the form \begin{equation}\label{nn_dense}
        G(x) = \sum_{j=1}^{N} a_j \sigma(y_j^T x + \theta_j) 
    \end{equation} are dense in $C(X)$ for all compact sets $X\subset \mathbb{R}^n$.
\end{theorem}
Later, the authors in \cite{leshno1993multilayer} showed that being a universal approximator is equivalent to having a non-polynomial activation function $\sigma$. Approximation theorems in the space of differentiable functions are established in \cite{li1996simultaneous} through the following theorem.
\begin{theorem}[Theorem 2.1 in \cite{li1996simultaneous}]\label{theorem3}
    Let X be a compact subset of $\mathbb{R}^n$ and $f\in \hat{C}^{m_1}(X)\cap\hat{C}^{m_2}(X)\cap\cdots\cap\hat{C}^{m_q}(X)$, where $m_i \in \mathbb{Z}^n_+$. Also, let $\sigma$ be any non-polynomial function in $C^n(\mathbb{R})$, where $n=\max\{|m_i|, 1\leq i\leq q\}$. Then for any $\epsilon > 0$, there is a neural network defined as in \eqref{nn_dense} such that 
    \begin{equation}
        \|D^kf - D^kG\|_{L^{\infty}(X)} < \epsilon,\quad k \leq m_i, \quad \text{ for some } i\quad, 1\leq i\leq q. \nonumber
    \end{equation}
\end{theorem}
The density argument on simultaneous approximations of multivariate functions and their partial derivatives have been applied to show the existence of a neural network that arbitrarily minimizes the PDE residual loss function, see \cite{NHM,sirignano2018dgm}. 

\section{The Deep Minimizing Movement Method}
We consider the following sequential minimization problem with time step $\tau$ :
\begin{equation}\label{JJKO}
    u_{k+1} = \argmin_{u \in X} \{\mathcal{F}(u) + \frac{d_X^2(u, u_{k})}{2\tau} \}, \quad u_{0} = u(0,x),
\end{equation}
where $X$ denotes either $L^2(\Omega)$, or $\mathcal{P}_2(\Omega)$ with $L^2$-Wasserstein distance. 

The main idea of this paper is to optimize \eqref{JJKO} by parametrizing $u$ by the neural network $u_\theta$. We present some theoretical results to support our method in the following sequel.

First of all, we show that there exists a neural network $u_\theta$ which minimizes the loss $\mathcal{F}(u_\theta) + \frac{d_{L^2(\Omega)}^2(u_\theta,u_k)}{2\tau}$ in \eqref{nn_dense} as much as we want in most cases.

\begin{theorem}\label{univl2thm}
Let $\Omega$ be a compact subset of $\mathbb{R}^n$ and $\phi:\Omega\times\mathbb{R}\times \mathbb{R}^n \rightarrow \mathbb{R}$ be a continuous function and $\tau>0$. Let $g,h$ be $C^k$-functions on $\Omega$. Define $\mathcal{I}(u):=\int_\Omega \phi(x,u(x),\nabla u(x)) dx  + \frac{d_{L^2(\Omega)}^2(u,h)}{2\tau}$ for every $C^k$-function $u$ on $\Omega$. Then, for $\epsilon>0$, there is a neural network $u_\theta$ defined as in \eqref{nn_dense} such that $||u_\theta-g||_{L^\infty(\Omega)} \leq \epsilon$ and $\mathcal{I}(u_\theta) - \mathcal{I}(g)\leq \epsilon$.
\end{theorem}

\begin{proof}
For notational convenience, let us define $\mathcal{J}(u):=\int_\Omega \phi(x,u(x),\nabla u(x)) dx$ for every $C^k$-function $u$ on $\Omega$, so that $\mathcal{I}(u) = \mathcal{J}(u) + \frac{d_{L^2(\Omega)}^2(u,h)}{2\tau}$ for all $u$.

Let $\epsilon>0$ be given and pick $\epsilon_1>0$ such that $(1+\frac{\epsilon_1}{2\tau})\cdot vol(\Omega)\cdot \epsilon_1 + \frac{1}{\tau}\cdot \sqrt{vol(\Omega)}\cdot \epsilon_1\cdot d_{L^2(\Omega)}(g,h)<\epsilon$.

Since $g$ is $C^k$ on $\Omega$, both $g$ and $\nabla g$ are bounded on $\Omega$. Thus, there exists $R>0$ such that the range of $g$ is contained in the interval $(-R,R)$ and the range of $\nabla g$ is contained in the open ball $B(0,R)$. Since $\phi$ is continuous, it is uniformly continuous on $\Omega\times[-R-\epsilon_1,R+\epsilon_1]\times \overline{B(0,R+\epsilon_1)}$. Thus, there exists $\delta$ such that $|x-y| + ||z-w||_2<\delta$ implies $|\phi(a,x,z) - \phi(a,y,w)|<\epsilon_1$ for any $(a,x,z),(a,y,w)\in \Omega\times[-R-\epsilon_1,R+\epsilon_1]\times \overline{B(0,R+\epsilon_1)}$. Then, by Theorem \ref{theorem3}, there exists a neural network $u_\theta$ satisfying $||u_\theta - g||_{L^\infty(\Omega)} + ||\nabla u_\theta - \nabla g||_{L^\infty(\Omega)}< \min\{\delta, \epsilon_1,\epsilon \}$, so that $||u_\theta - g||_{L^\infty(\Omega)}\leq \epsilon$.

Next, we pick $x\in \Omega$. Since $||\nabla u_\theta(x)||_2 - ||\nabla g(x)||_2 \leq ||\nabla u_\theta(x) - \nabla g(x)||_2 \leq ||\nabla u_\theta - \nabla g||_{L^\infty(\Omega)} <\epsilon_1$, we know that $\nabla u_\theta(x) \in \overline{B(0,R+\epsilon_1)}$. Analogously, $u_\theta(x)\in [-R-\epsilon_1, R+\epsilon_1]$. Therefore, both $(u_\theta(x),\nabla u_\theta(x))$ and $(g(x),\nabla g(x))$ are elements of $[-R-\epsilon_1, R+\epsilon_1]\times \overline{B(0,R+\epsilon_1)}$. Since $|u_\theta(x)-g(x)|+||\nabla u_\theta(x) - \nabla g(x)||_2<\delta$, we conclude that $|\phi(x,u_\theta(x),\nabla u_\theta(x)) - \phi(x,g(x),\nabla g(x))|<\epsilon_1$. Since this holds for every $x$, we have

\begin{align*}
    |\mathcal{J}(u_\theta) - \mathcal{J}(g)|&= \bigg| \int_\Omega \phi(x,u_\theta(x), \nabla u_\theta(x)) - \phi(x,g(x),\nabla g(x)) dx\bigg|\\
    &\leq \int_\Omega |\phi(x,u_\theta(x), \nabla u_\theta(x)) - \phi(x,g(x),\nabla g(x))| dx \\
    &\leq \epsilon_1\cdot vol(\Omega)
\end{align*}

Now, note that $d_{L^2(\Omega)}^2(u_\theta,g)=\int_\Omega |u_\theta-g|^2 dx\leq vol(\Omega)\cdot ||u_\theta-g||_{L^\infty(\Omega)}^2\leq vol(\Omega)\cdot \epsilon_1^2$. Thus, we have

\begin{align*}
    d_{L^2(\Omega)}^2(u_\theta,h) - d_{L^2(\Omega)}^2(g,h) &\leq \bigg(d_{L^2(\Omega)}(u_\theta, g) + d_{L^2(\Omega)}(g,h)\bigg)^2 - d_{L^2(\Omega)}^2(g,h) \\
    &= d_{L^2(\Omega)}^2(u_\theta,g) + 2\cdot d_{L^2(\Omega)}(u_\theta,g) \cdot d_{L^2(\Omega)}(g,h) \\
    &\leq vol(\Omega)\cdot \epsilon_1^2 + 2\cdot \sqrt{vol(\Omega)}\cdot \epsilon_1\cdot d_{L^2(\Omega)}(g,h)
\end{align*}

Hence,

\begin{align*}
    \mathcal{I}(u_\theta) -\mathcal{I}(g)&\leq (1+\frac{\epsilon_1}{2\tau})\cdot vol(\Omega)\cdot \epsilon_1 + \frac{1}{\tau}\cdot \sqrt{vol(\Omega)}\cdot \epsilon_1\cdot d_{L^2(\Omega)}(g,h) \\
    &\leq \epsilon.
\end{align*}

\end{proof}

Theorem \ref{univl2thm} shows that if we take $g$ as the global minimizer $u_{k+1}$ of \eqref{nn_dense} with $h=u_k$, there always exist a neural network $u_\theta$ such that $\mathcal{I}(u_\theta)\leq \mathcal{I}(u_{k+1}) + \epsilon$ with $||u_\theta - u_{k+1}||_{L^\infty(\Omega)}\leq \epsilon$. Therefore, we can expect that we can find a neural network solution approximating $u_{k+1}$ by optimizing the parametric form of \eqref{nn_dense}.

Next, we will justify our method for $\mathbb{W}_2$-gradient flows. Yet, we need a variation of the standard universal approximation theorem for probability density functions before we prove the analogous result of Theorem \ref{univl2thm} for $\mathbb{W}_2$ spaces. This is because we will approximate a given probability density function by neural networks which behave as if probability density functions. We start by proving that any \textit{non-negative} $C^k$-function can be approximated by a \textit{non-negative} neural network.

\begin{theorem}\label{universal_nonneg}
Let $X,f,\sigma$ and $m_i$'s be given as in Theorem \ref{theorem3} with an additional assumption that $f$ is non-negative. Then, for any $\epsilon>0$, there is a \textit{non-negative} neural network defined as in $\eqref{nn_dense}$ such that
\begin{equation}
        \|D^kf - D^kG\|_{L^{\infty}(X)} < \epsilon,\quad k \leq m_i, \quad \text{ for some } i, 1\leq i\leq q. \nonumber
    \end{equation}
\end{theorem}

\begin{proof}
Fix $\epsilon>0$ and define $\tilde{f}:=f + \epsilon$. Then, by Theorem \ref{theorem3}, there exists a neural network $\tilde{G}$ of the form $\sum_{j=1}^{N} a_j \sigma(y_j^T x + \theta_j)$ such that 
\begin{equation}
        \|D^k\tilde{f} - D^k\tilde{G}\|_{L^{\infty}(X)} < \epsilon/2,\quad k \leq m_i, \quad \text{ for some } i, 1\leq i\leq q. \nonumber
\end{equation}

Since $\tilde{f}(x) - \tilde{G}(x) \leq ||\tilde{f}-\tilde{G}||_{L^\infty(\Omega)} <\epsilon /2$ for any $x\in X$, we have $\tilde{G}(x)>\tilde{f} -\epsilon/2 =f(x)+\epsilon/2 \geq \epsilon/2$ for any $x\in X$.

Since $\sigma$ is non-polynomial, there exists $\theta_{N+1}\in\mathbb{R}$ such that $\sigma(\theta_{N+1})\neq 0$. Now, we define $y_{N+1}:=0$ and $a_{N+1}:=\frac{-\epsilon}{2\cdot \sigma(\theta_{N+1})}$ and $G(x):=\sum_{j=1}^{N+1} a_j \sigma(y_j^T x + \theta_j)$ for every $x\in X$, so that $G$ is a neural network defined as in \eqref{nn_dense}.

Now, we pick $x\in X$. Then,

\begin{align}\label{nonnegativity_middle}
    G(x)&=\tilde{G}(x) + a_{N+1}\sigma(y_{N+1}^Tx+\theta_{N+1}) \nonumber\\
    &=\tilde{G}(x) + a_{N+1}\sigma(\theta_{N+1}) \nonumber\\
    &=\tilde{G}(x) -\epsilon/2 \geq 0.
\end{align}

\noindent Therefore, we have

\begin{align*}
    f(x)-G(x)&=f(x)-\tilde{G}+\epsilon/2 \\
    &=f(x) + \epsilon - \tilde{G}(x) - \epsilon/2 \\
    &= \tilde{f}(x) - \tilde{G}(x) - \epsilon/2.
\end{align*}
    
\noindent Thus, $|f(x)-G(x)|\leq |\tilde{f}(x)-\tilde{G}(x)|+\epsilon/2 < \epsilon$ and $G(x)\geq 0$ for every $x\in X$ by \eqref{nonnegativity_middle}, meaning that $||f-G||_{L^\infty(\Omega)}<\epsilon$ and $G$ is a non-negative neural network as in \eqref{nn_dense}. Since $\tilde{f}=f+\epsilon$ and $G = \tilde{G} -\epsilon/2$,  $D^kf =D^k\tilde{f}$ and $D^kG=D^k\tilde{G}$ for every $k\neq 0$. Consequently, $G$ is the desired non-negative neural network.

\end{proof}

Now, we prove that there exists a probability density function $\widetilde{u_\theta}$ which can be produced by manipulating a neural network, which approximates a given density $f$ as much as we want.

\begin{theorem}\label{probdensityuniversal}
Let $X,f,\sigma$ and $m_i$'s be given as in Theorem \ref{theorem3} with an additional assumption that $f$ is a probability density function whose domain is $X$. Then, for any $\epsilon>0$, there is a non-negative neural network $G$ defined as in \eqref{nn_dense} such that 
\begin{equation}
        \|f - \tilde{G}\|_{L^{\infty}(X)} < \epsilon, \nonumber
\end{equation}
where $\tilde{G} = \frac{1}{\int_X G} G$.
\end{theorem}

\begin{proof}
    Fix $\epsilon >0$. Then, take $\epsilon_0>0$ and $0<\epsilon_1<1$ such that $\epsilon_0\cdot ||f||_{L^\infty(\Omega)} + \frac{\epsilon_0^2}{vol(X)} < \epsilon/2$ and $\max\{ \frac{\epsilon_1}{1-\epsilon_1}, \frac{\epsilon_1}{1+\epsilon_1}, \epsilon_1\} < \epsilon_0$.
    
    Then, by Theorem \ref{universal_nonneg}, there exists a non-negative neural network $G$ defined as in \eqref{nn_dense} such that $||f-G||_{L^\infty(\Omega)} < \min\{\frac{\epsilon_1}{vol(X)}, \epsilon/2 \}$. Thus, $||G||_{L^\infty(\Omega)} \leq ||f||_{L^\infty(\Omega)} + \frac{\epsilon_1}{vol(X)} < ||f||_{L^\infty(\Omega)} + \frac{\epsilon_0}{vol(X)}$. 
    Since
    \begin{align*}
        \bigg|1-\int_X G\bigg| &= \bigg|\int_X f - \int_X G\bigg| = \bigg|\int_X (f-G)\bigg| \leq \int_X |f-G|\\
        & \leq vol(X)\cdot ||f-G||_{L^\infty(\Omega)}< vol(X)\cdot \frac{\epsilon_1}{vol(X)}=\epsilon_1,
    \end{align*}
    
    \noindent we have $1-\epsilon_1 < \int_X G < 1+\epsilon_1$.
    Therefore, $-\epsilon_0<-\frac{\epsilon_1}{1-\epsilon_1} < 1-\frac{1}{\int_X G} < \frac{\epsilon_1}{1+\epsilon_1} < \epsilon_0$, leading to $\bigg| 1-\frac{1}{\int_X G}\bigg| \leq \epsilon_0$.
    
    Hence,
    \begin{align*}
        \bigg|\bigg| G- \frac{G}{\int_X G} \bigg|\bigg|_{L^\infty(\Omega)} &\leq ||G||_{L^\infty(\Omega)}\cdot \bigg| 1 - \frac{1}{\int_X G} \bigg| \\
        &\leq ||G||_{L^\infty(\Omega)} \cdot \epsilon_0 \\
        &\leq (||f||_{L^\infty(\Omega)} + \frac{\epsilon_0}{vol(X)})\cdot \epsilon_0 \\
        &= \epsilon_0\cdot ||f||_{L^\infty(\Omega)} + \frac{\epsilon_0^2}{vol(X)} \\
        &< \epsilon/2.
    \end{align*}
    
    Therefore, we have
    \begin{align*}
        \bigg|\bigg|f-\frac{G}{\int_X G}\bigg|\bigg|_{L^\infty(\Omega)} \leq ||f-G||_{L^\infty(\Omega)} + \bigg|\bigg| G- \frac{G}{\int_X G} \bigg|\bigg|_{L^\infty(\Omega)} <\epsilon/2 + \epsilon/2 = \epsilon.
    \end{align*}
\end{proof}

Throughout, we identify absolutely continuous measures with their densities, $d\rho(x)=\rho(x)dx$. We will now show that there exists a non-negative neural network $u_\theta$ which minimizes the loss $\mathcal{F}(\widetilde{u_\theta})+\frac{\mathbb{W}_2^2(\widetilde{u_{\theta}},u_k)}{2\tau}$ in \eqref{JJKO} as much as we want in most cases, where $\widetilde{u_\theta}:=\frac{1}{\int_\Omega u_\theta}u_\theta$.

\begin{lemma}\label{w2distproof}
Assume that $\Omega$ is a compact subset of $\mathbb{R}^n$, and let $u, \rho$ be absolutely continuous probability distributions on $\Omega$ with finite second moments such that $||u-\rho||_{L^\infty(\Omega)} \leq \epsilon$. Then, $\mathbb{W}_2(u,\rho) \leq C\cdot \epsilon^{1/2}$ for some constant $C$ which only depends on $\Omega$. 
\end{lemma}

\begin{proof}
Let $\gamma$ be a joint distribution of $u$ and $\rho$. Then, we have

\begin{align*}
\mathbb{W}_2(u,\rho) &= \bigg( \int_\Omega ||x-y||_2^2 d\gamma(x,y) \bigg)^{1/2} \\
&= \bigg( \int_\Omega ||x-y||_2\cdot ||x-y||_2 d\gamma(x,y) \bigg)^{1/2} \\
&\leq diam(\Omega)^{\frac{1}{2}} \cdot \bigg( \int_\Omega ||x-y||_2 d\gamma(x,y) \bigg)^{1/2}.
\end{align*}

\noindent Since $\gamma$ was arbitrary, we obtain

\begin{equation}\label{KRBBB}
    \mathbb{W}_2(u,\rho)\leq diam(\Omega)^{\frac{1}{2}}\cdot \bigg(\inf_{\gamma\in \Pi(u,\rho)} \int_\Omega ||x-y||_2 d\gamma(x,y) \bigg)^{\frac{1}{2}},
\end{equation}
where $\Pi(u,\rho)$ is the collection of joint distributions of $u$ and $\rho$.

The term $\inf_{\gamma\in \Pi(u,\rho)} \int_\Omega ||x-y||_2 d\gamma(x,y)$ is called the Wasserstein-1 metric and is denoted by $\mathbb{W}_1(u,\rho)$. Now, by the Kantorovich-Rubinstein duality theorem, we have

\begin{equation}\label{KRB}
    \mathbb{W}_1(u,\rho) = \sup \biggl\{ \int_\Omega f(x) d(u-\rho)(x) : f\in C(\Omega) \text{ with } ||f||_{Lip} \, \leq 1 \biggr\},
\end{equation}
where $||\cdot||_{Lip}$ denotes the minimal Lipschitz constant.

Define the set $\mathcal{E}:=\biggl\{ \int_\Omega g(x) d(u-\rho)(x): g\in C(\Omega) \text{ with } ||g||_{L^\infty(\Omega)}\leq 1\biggr\}$. Let $f\in C(\Omega)$ with $||f||_{Lipz}\leq 1$, and pick a point $x_0\in \Omega$. Now, define $g:=\frac{1}{diam(\Omega)}\cdot (f- f(x_0))$. Since $|f(x)-f(x_0)|\leq ||x-x_0||_2\leq diam(\Omega)$ for every $x$, we have $||g||_{L^\infty(\Omega)} \leq 1$. Since 
\begin{align*}
    \int_\Omega g(x) d(u-\rho)(x) &= \frac{1}{diam(\Omega)}\int_\Omega (f(x) - f(x_0)) d(u-\rho)(x) \\
    &= \frac{1}{diam(\Omega)}\int_\Omega f d(u-\rho)(x),
\end{align*}
we have $\frac{1}{diam(\Omega)} \int_\Omega f(x)d(u-\rho)(x) = \int_\Omega g(x) d(u-\rho)(x)\in \mathcal{E}$. Since $f$ was arbitrary, by \eqref{KRB}, we have

\begin{equation}
\mathbb{W}_1(u,\rho) \leq diam(\Omega)\cdot \sup\mathcal{E}.
\end{equation}

\noindent Combining with \eqref{KRBBB}, we get
\begin{equation}
    \mathbb{W}_2(u,\rho)\leq diam(\Omega)\cdot (\sup \mathcal{E})^{\frac{1}{2}},
\end{equation}

Next, consider $g\in C(\Omega)$ with $||g||_{L^\infty(\Omega)}\leq 1$. Then,

\begin{align*}
    \int_\Omega g(x)d(u-\rho)(x) &= \int_\Omega g(x)(u(x)-\rho(x)) dx \\
    &\leq \int_\Omega |g(x)|\cdot |u(x)-\rho(x)| dx \\
    &\leq \int_\Omega |u(x)-\rho(x)| dx \\
    &\leq vol(\Omega)\cdot ||u-\rho||_{L^\infty(\Omega)}\\
    &\leq vol(\Omega)\cdot \epsilon.
\end{align*}

\noindent Since this holds for all such $g$, we have $\sup \mathcal{E}\leq vol(\Omega)\cdot\epsilon$.

Consequently, we have $\mathbb{W}_2^2(u,\rho)\leq C \cdot \epsilon$ where $C$ only depends on $\Omega$.
\end{proof}

\begin{theorem}\label{w2jkoproof}
Let $\Omega$ be a compact subset of $\mathbb{R}^n$ and $\phi:\Omega\times[0,1]\rightarrow \mathbb{R}$ be a continuous function and $\tau>0$. Let $\rho,\pi$ be absolutely continous $C^k$ probability distributions on $\Omega$ with finite second moments.
Define $\mathcal{I}(u):= \int_\Omega \phi(x,u(x)) dx + \frac{\mathbb{W}_2^2(u,\pi)}{2\tau}$ for every $C^k$ probability distribution $u$ on $\Omega$ with finite second moment. Then, for $\epsilon>0$, there is a non-negative neural network $u_\theta$ defined as in \eqref{nn_dense} such that $||\widetilde{u_\theta} - \rho||_{L^\infty(\Omega)}\leq\epsilon$ and $\mathcal{I}(\widetilde{u_\theta}) - \mathcal{I}(\rho)\leq\epsilon$, where $\widetilde{u_\theta} = \frac{1}{\int_\Omega u_\theta}u_\theta$.

\end{theorem}

\begin{proof}
    Let $\epsilon>0$ be given, and $C$ be the constant given in Lemma \ref{w2distproof} which only depends on $\Omega$. Let $\epsilon_0>0$ be a constant such that $\epsilon_0 + \frac{1}{2\tau}(\epsilon_0 + 2\cdot \sqrt{\epsilon_0}\cdot \mathbb{W}_2(\rho,\pi)) < \epsilon$. Since $\Omega\times[0,1]$ is compact, $\phi$ is uniformly continuous. Therefore, there exists $\delta$ such that $|x-y|\leq\delta$ implies $|\phi(z,x)-\phi(z,y)|\leq \frac{1}{vol(\Omega)}\cdot \epsilon_0$, for every $z\in\Omega$ and $x,y\in [0,1]$. Now, by Theorem \ref{probdensityuniversal}, there is a non-negative neural network $u_\theta$ satisfying $||\widetilde{u_\theta} - \rho||_{L^\infty(\Omega)}<\min \{\delta, \frac{1}{C^2}\cdot \epsilon_0, \epsilon\}$. Therefore, we obtain 
    
\begin{align*}
    \bigg|\int_\Omega \phi(x,\widetilde{u_\theta}(x)) dx - \int_\Omega \phi(x,\rho(x)) dx\bigg|
    &\leq \int_\Omega |\phi(x,\widetilde{u_\theta}(x)) - \phi(x,\rho(x))| dx \\
    &\leq vol(\Omega)\cdot \frac{1}{vol(\Omega)}\cdot\epsilon_0 \\
    &=\epsilon_0
\end{align*}

Since $||\widetilde{u_\theta} - \rho||_{L^\infty(\Omega)}<\frac{1}{C^2}\cdot \epsilon_0$, by Lemma \ref{w2distproof}, we have  $\mathbb{W}_2(\widetilde{u_\theta},\rho)\leq \sqrt{\epsilon_0}$.
Therefore, 
\begin{align*}
    \mathbb{W}_2^2(\widetilde{u_\theta},\pi) - \mathbb{W}_2^2(\rho,\pi) &\leq \bigg(\mathbb{W}_2(\widetilde{u_\theta},\rho) + \mathbb{W}_2(\rho,\pi)\bigg)^2 - \mathbb{W}_2^2(\rho, \pi) \\
    &= \mathbb{W}_2^2(\widetilde{u_\theta},\rho) + 2\cdot \mathbb{W}_2(\widetilde{u_\theta},\rho)\cdot \mathbb{W}_2(\rho,\pi) \\
    &\leq \epsilon_0 + 2\cdot \sqrt{\epsilon_0} \cdot \mathbb{W}_2(\rho,\pi) \\
    &\leq 2\tau \cdot (\epsilon - \epsilon_0)
\end{align*}

Consequently, we have

\begin{align*}
    \mathcal{I}(\widetilde{u_\theta}) - \mathcal{I}(\rho) &= \int_\Omega \phi(\widetilde{u_\theta}(x)) - \phi(\rho(x)) dx + \frac{1}{2\tau}\cdot (\mathbb{W}_2^2(\widetilde{u_\theta},\pi) - \mathbb{W}_2^2(\rho,\pi)) \\
    &\leq \epsilon_0 + (\epsilon -\epsilon_0) \\
    &=\epsilon
\end{align*}

\end{proof}

Theorem \ref{w2jkoproof} shows that if we take $\rho$ as the global minimizer $u_{k+1}$ of \eqref{nn_dense} with $\pi=u_k$, there always exist a neural network $u_\theta$ such that $\mathcal{I}(\widetilde{u_\theta}) \leq \mathcal{I}(u_{k+1}) + \epsilon$ with $||\widetilde{u_\theta} - u_{k+1}||_{L^\infty(\Omega)}\leq \epsilon$, where $\widetilde{u_\theta}=\frac{1}{\int u_\theta} u_\theta$. Therefore, we can expect that we can find a neural network solution approximating $u_{k+1}$ by optimizing the parametric form of \eqref{nn_dense}.

Lastly, we present a result which shows that the increment of the error at each update is bounded above.

\begin{theorem}\label{theorem4}
Let $X$ be a metric space and $\mathcal{F}:X\rightarrow \mathbb{R}\cup\{+\infty\}$ be a lower semi-continuous function with some lower bounds to guarantee the existence of minimizers of \eqref{JJKO} for small $\tau$. Given $u_0\in X$ with $\mathcal{F}(u_0)<\infty$, let $\{u_k\}$ be the true minimizers of \eqref{JJKO} for each $k$, and let $\{u_{\theta,k}\}$ be a sequence of neural networks obtained by solving the optimization problem \eqref{JJKO} iteratively with a natural assumption that $\mathcal{F}(u_{\theta,k+1})+\frac{d_X^2(u_{\theta,k+1},u_{\theta,k})}{2\tau}\leq \mathcal{F}(u_{\theta,k}) + \frac{d_X^2(u_{\theta,k}, u_{\theta,k})}{2\tau}$ for each $k$. Then, we obtain
\begin{equation*}
    d_X(u_{\theta,k+1},u_{k+1}) \leq d_X(u_{\theta,k},u_k) + C\tau^{1/2}
\end{equation*}
for some constant $C$ depending on $\mathcal{F}$ and initial $u_0$.

\end{theorem}

\begin{proof}
Firstly, since $u_{k+1}$ is a minimizer of \eqref{JJKO}, we get 
\begin{equation*}
\mathcal{F}(u_{k+1})+\frac{d_X^2(u_{k+1},u_{k})}{2\tau} \leq \mathcal{F}(u_{k})+\frac{d_X^2(u_{k},u_{k})}{2\tau} = \mathcal{F}(u_k)
\end{equation*}

Therefore, we get $\mathcal{F}(u_{k+1})\leq \mathcal{F}(u_k)$ for each $k$. Since this is valid for every $k$, we have $\mathcal{F}(u_k)\leq \mathcal{F}(u_0)<+\infty$ for each $k$. Analogous arguments apply to $u_{\theta,k}$, and we conclude that $\mathcal{F}(u_{\theta,k})<+\infty$ for each $k$.

Again, by the definition of $u_k$'s, we get
\begin{align*}
    d_X^2(u_{k}, u_{k+1}) \leq 2\tau (\mathcal{F}(u_{k}) - \mathcal{F}(u_{k+1})) \leq 2\tau (\mathcal{F}(u_{0}) - \inf\mathcal{F}) = 2A\tau,
\end{align*}
where $A=A(u_0, \mathcal{F}) \coloneqq (\mathcal{F}(u_{0}) - \inf\mathcal{F})$ only depends on the functional $\mathcal{F}$ and the initial $u_0$. Analogously, by the assumption, we get $d_X^2(u_{\theta,k}, u_{\theta,k+1}) < 2A\tau$. Then,
\begin{align*}
    d_X(u_{\theta, k+1}, u_{k+1}) &\leq d_X(u_{\theta, k+1}, u_{\theta, k}) + d_X(u_{\theta, k}, u_{k}) + d_X(u_{k}, u_{k+1}) \\
    &< \sqrt{2A\tau} + d_X(u_{\theta, k}, u_{k}) + \sqrt{2A\tau} = d_X(u_{\theta, k}, u_{k}) + C\tau^{1/2},
\end{align*}
where $C=2\sqrt{2A}$.
\end{proof}

In the remaining of this section, we give a detailed description of the framework of our method for the spaces $L^2(\Omega)$ and $\mathcal{P}_2(\Omega)$ with $L^2$-Wasserstein distance.

\subsection{Case when \texorpdfstring{$\boldsymbol{X=L^2(\Omega)}$}{TEXT}%
}\label{l2estimate}

If $X=L^2(\Omega)$ and the functional $\mathcal{F}(u)$ is given in the form $\int \phi(u,\nabla u)$, this integration can be evaluated by Monte Carlo estimate over mini-batches from the uniform distribution on $\Omega$ if $\Omega$ is bounded. If $\Omega=\mathbb{R}^n$, we can still approximate this integration by restricting the domain to a sufficiently large cropped bounded subset, i.e. $\int_{\mathbb{R}^n} \phi(u,\nabla u) \approx \int_{[-M,M]^n} \phi(u,\nabla u)$. Likewise, the distance $d_{L^2(\Omega)}^2(u,v)=\int_\Omega ||u-v||^2$ can also be evaluated by Monte Carlo estimate in the same manner. Now, by Theorem \ref{univl2thm}, it would suffice to minimize
\begin{equation}
    \argmin_\theta \{ \int \phi (u_\theta, \nabla u_\theta) + \frac{d_{L^2(\Omega)}^2(u_\theta, u_k)}{2\tau}\},
    \label{eq:L2_target}
\end{equation}
for a given $u_k$ to obtain the minimizer $u_{\theta^{*(k+1)}}$, which is an approximation of the minimizer of (\ref{JJKO}). Then, we take the optimized $u_{\theta^{*(k+1)}}$ in place of $u_k$ and repeat this process iteratively to derive the corresponding time-discretized gradient flow.

\subsection{Case when \texorpdfstring{$\boldsymbol{X=\mathcal{P}_2(\Omega)}$}{TEXT}%
}

If $X=\mathcal{P}_2(\Omega)$ and the functional $\mathcal{F}(u)$ is given as a linear combination of potential energy $\mathcal{V}$ and interaction energy $\mathcal{W}$ and internal energy $\mathcal{E}$, there is no problem approximating $\mathcal{F}(u)$ by Monte Carlo estimate as described in Section \ref{l2estimate}. Moreover, by Theorem \ref{probdensityuniversal}, any probability density function can be approximated by manipulating a neural network $u_\theta$ having non-negative outputs. Therefore, if we define $\widetilde{u_\theta}:= \frac{1}{\int u_\theta}u_\theta$ for non-negative neural networks $u_\theta$ so that $\int \widetilde{u_\theta} = 1$, it suffices to minimize
\begin{equation}\label{W2mms}
    \argmin_\theta \{ \mathcal{F} (\widetilde{u_\theta}) + \frac{\mathbb{W}_2^2(\widetilde{u_\theta}, u_k)}{2\tau}\},    
\end{equation}
for a given $u_k$ to obtain the minimizer $\widetilde{u_{\theta^{*(k+1)}}}$, which is an approximation of the minimizer of (\ref{JJKO}). As in the case $X=L^2(\Omega)$, we take the optimized $\widetilde{u_{\theta^{*(k+1)}}}$ in place of $u_k$ and repeat this process iteratively to derive the corresponding time-discretized gradient flow.

However, unlike the case $X=L^2(\Omega)$, there is no direct easy way to evaluate the squared metric, i.e. $\mathbb{W}_2^2(u,\rho)$, when minimizing the objective (\ref{W2mms}). There are many algorithms to approximate the $L^2$-Wasserstein distance, but we take an idea proposed in \cite{korotin2019wasserstein} for the reasons described in Section \ref{optimaltransport}. In other words, we approximate it by solving the optimization problem (\ref{W2objective}) at first and then by evaluating (\ref{W2approx}) using the optimal $\psi_{\theta^*}$ and $\overline{\psi_{w^*}}$. All the integrations in formulae (\ref{W2approx}) and (\ref{W2objective}) are again evaluated by Monte Carlo as described in Section \ref{l2estimate}. Moreover, the non-negativity assumption on neural networks can be easily satisfied by choosing non-negative activation functions such as Softplus \cite{dugas2001incorporating}, ReLU \cite{nair2010rectified}, ELU \cite{clevert2015fast}, etc.

\section{Experiments}
In this section, we show through various kinds of examples, that one can accurately approximate the solutions of PDEs which have gradient flow structures via neural networks. We divide this section into two subsections, each of which is devoted to gradients flows for the case $L^2$ and the case $\mathbb{W}_2$ respectively. For the $L^2$ gradient flow case, we demonstrate the scalability of the proposed method to the high-dimensional problems through the heat equation. Throughout this section, we apply the domain truncation technique for the unbounded domains (in our cases $\mathbb{R}^n$) as applied in the literature (See, for example, \cite{hwang2020trend,cho2021traveling}). We compute the $L^2$ error of a neural network solution $u_{NN}$ by $\|u - u_{NN}\|_{L^2}$ and the relative $L^2$ error by $\frac{\|u - u_{NN}\|_{L^2}}{\|u\|_{L^2}}$, where $u$ denotes either an analytic solution or a numerical solution. 

\paragraph{Motivating Example} We first consider a simple motivating example with the $L^2$ energy functional $\frac{\kappa}{2}\int_{\Omega} u^2 dx$. The corresponding minimization problem reads as :
\begin{equation}\label{motivating}
    \argmin_{\theta} \{\frac{\kappa}{2}\int_{\Omega} u_{\theta}^2 dx + \frac{d^2_{L^2(\Omega)}(u_{\theta}, u_{k})}{2\tau}\},
\end{equation} where $u_{\theta}$ is a neural network solution and $\theta$ denotes the set of parameters. We set an initial condition $u(0,x) = \sin(x)$, $\Omega = [-\pi, \pi]$, then the corresponding solution of the PDE is $u(t,x) = e^{-\kappa t}\sin(x)$. In Figure \ref{fig:motivating}, we demonstrate the minimizer $u_{\theta, k+1}$ approximately satisfies the Euler-Lagrange equation of \eqref{motivating}, i.e., \begin{equation}
    \frac{u_{\theta, k+1} - u_{\theta, k}}{\tau} = - \nabla F(u_{k+1}),\nonumber
\end{equation} during $40$ timesteps with $\tau= 0.01$.

\begin{figure}[h]
    \centering
    \includegraphics[width=0.9\textwidth]{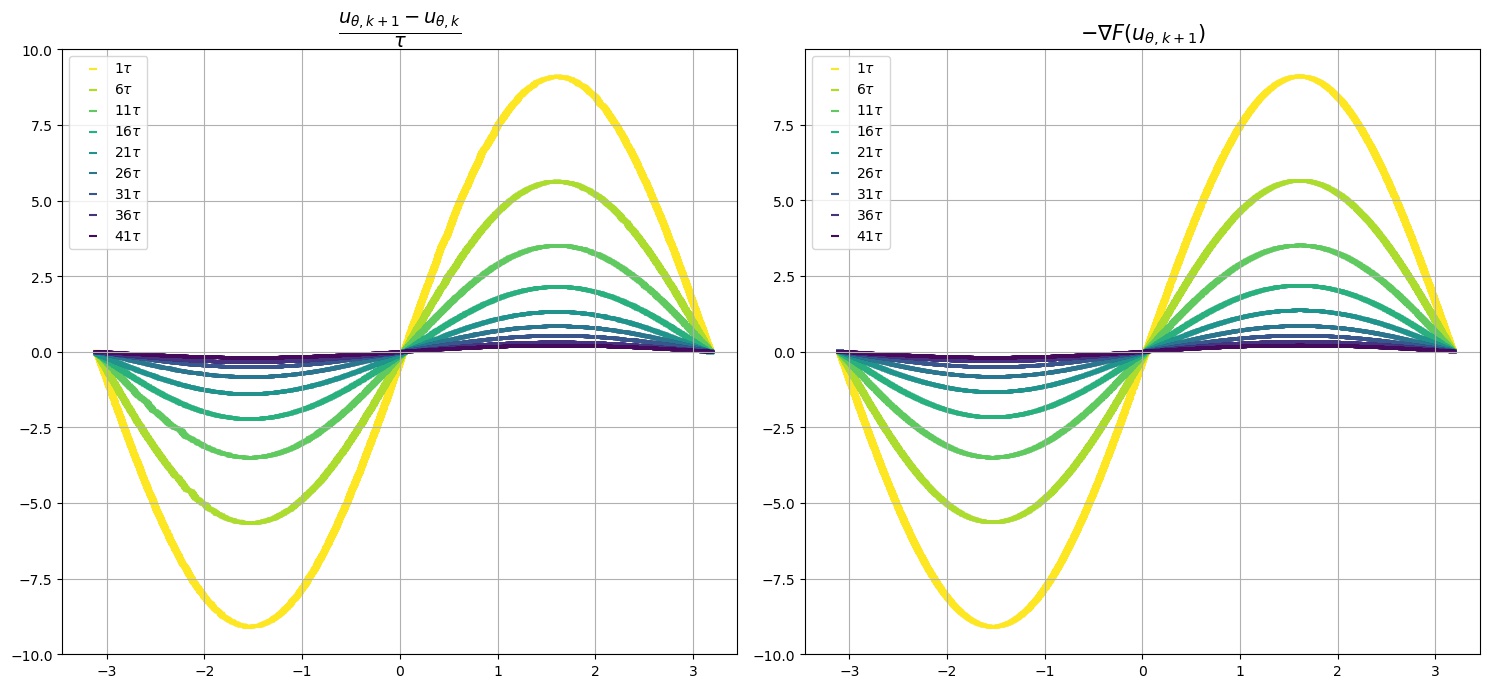}
    \caption{Left: Scatter plot for the values of $\frac{u_{\theta, k+1} - u_{\theta, k}}{\tau}$ on the uniform samples. Right: Scatter plot for the values of $-\nabla F(u_{\theta, k+1})$ on the same sample points.}
    \label{fig:motivating}
\end{figure}

\subsection{\texorpdfstring{$\boldsymbol{L^2}$}{TEXT}%
-gradient flows}
In this section, we consider two examples of $L^2$-gradient flows, the heat equation and the Allen-Cahn equation. 

\subsubsection{Heat equation: 2-dimensional, 8-dimensional}\label{subsubsec:2dheat}
The heat equation
\begin{equation}\label{heat}
    \begin{cases}
        \partial_t u = \kappa\Delta u \quad \text{in } \Omega, \\
        \nabla u \cdot \hat{\boldsymbol{n}}=0 \quad \text{on } \partial \Omega,
    \end{cases}
\end{equation}
where $\boldsymbol{n}$ denotes the unit normal vector, is the $L^2$-gradient flow of the Dirichlet energy $\mathcal{F}(u):=\frac{\kappa}{2}\int_\Omega ||\nabla u||^2 dx$.

For the special case $\Omega = [-a, a]^n$, where the domain is given by an $n$-dimensional cube, the ground-truth solution of (\ref{heat}) can be computed. Let the initial condition be given by $u(0,x) = \sum_{[i]} A_{[i]} \prod_k \cos(\frac{i_k\pi x_k}{a})$, where $[i]=(i_1,i_2,...,i_n)$ denotes the multi-index and $A_{[i]}$ denotes the coefficient of $\prod_k \cos(\frac{i_k\pi x_k}{a})$. Assuming the convergence of the initial condition $u(0,x)$, then
\begin{align}
    u(t,x) &= \sum_{[i]} A_{[i]} \biggr(\prod_k \cos(\frac{i_k\pi x_k}{a})\biggr)e^{-\kappa\mu_{[i]}^2t},\quad\mu_{[i]}^2 = \pi^2\sum_k (\frac{i_k}{a})^2
    \label{heat_known_sol}
\end{align}
solves (\ref{heat}) and satisfies the boundary condition. 

\paragraph{2-dimensional case}
In the experiment for 2-dimensional heat equation, we set $\kappa=0.1$ and $a = \pi/4$. The domain $\Omega$ is given by $[-\pi/4, \pi/4]^2$. For an initial condition $u_0(x) = 4-0.5\cos{4x_1}+0.6\cos{4x_2}+\cos{8x_1}\cos{4x_2}+\cos{8x_1}\cos{8x_2}$, its ground-truth solution $u(t,x)$ is computed as in (\ref{heat_known_sol}). A size of the time-step for the minimizing movement scheme is given by $\tau = 0.005$. 

\paragraph{Model training details for 2d case}
We use a 2-layer fully connected neural network with 256 hidden units for the model that approximates the solution $u(x)$. A hyperbolic tangent (tanh) activation function is applied. We choose the learning rate as $10^{-3}$ with an Adam optimizer \cite{kingma2014adam}. While training the neural network model, we uniformly picked 10,000 samples from $\Omega = [-\pi/4,\pi/4]^2$ and computed the minimization target (\ref{eq:L2_target}) by the Monte-Carlo approximation.

\paragraph{Results - 2d case} 
The approximated results are illustrated in Figure \ref{fig:heat_2d}. We visualized the ground-truth solution and the approximated neural network solution at the time $t=2\tau, 20\tau, 40\tau$. Overall, the deep minimizing movement scheme achieved a good result on the 2d heat equation. Our approximated neural network solution achieved the relative $L^2$ error around $3.15\times10^{-3}$ at $t=0.2=40\tau$.

\begin{figure}[h]
    \centering
    \includegraphics[width=0.95\textwidth]{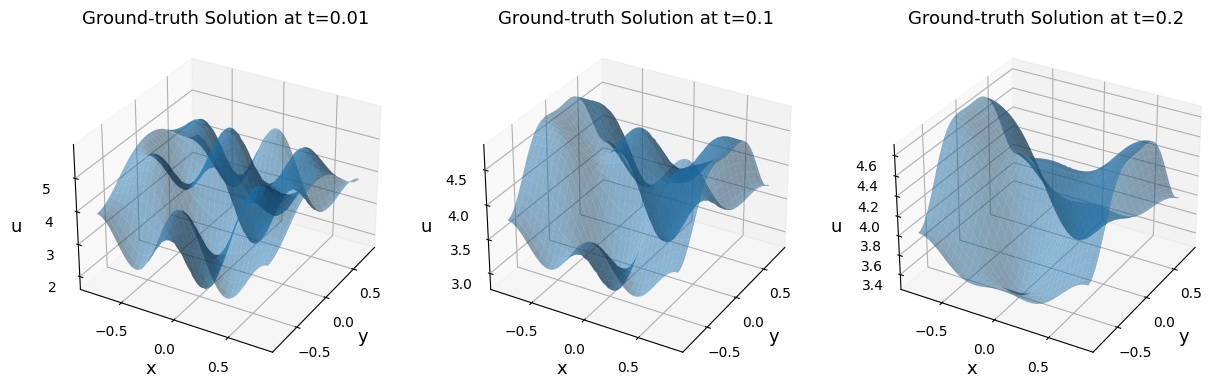}
    \includegraphics[width=0.95\textwidth]{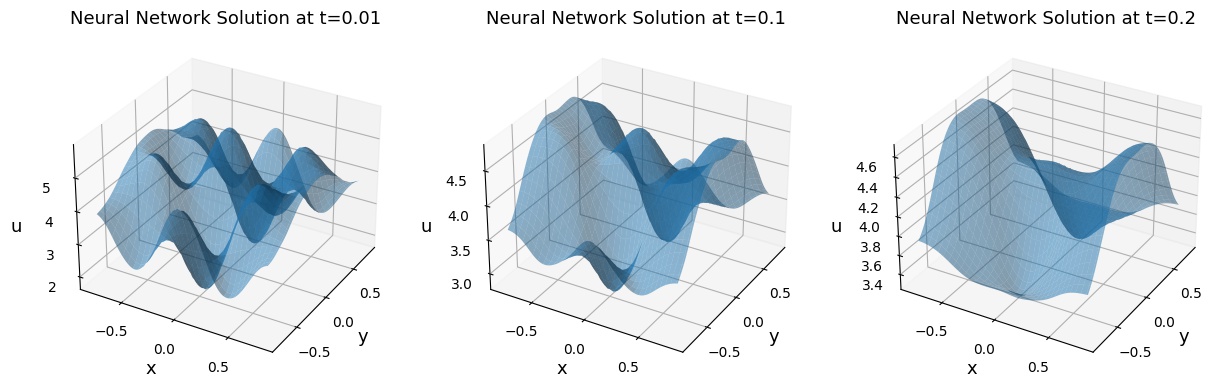} \\
    \vspace{0.4cm}
    \includegraphics[width=0.95\textwidth]{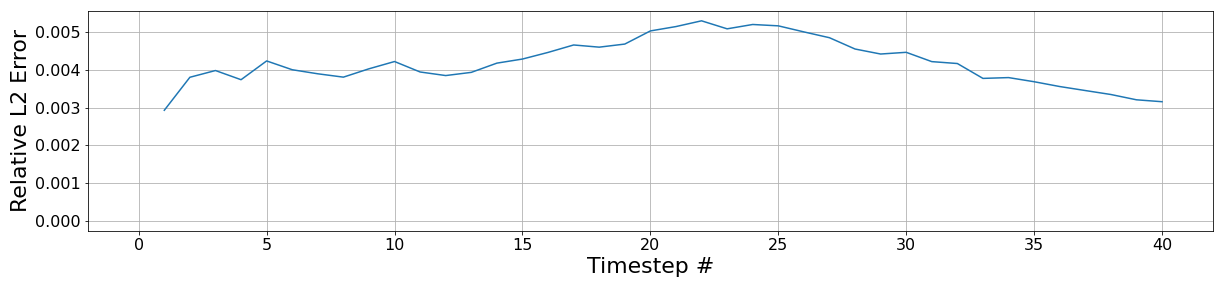}
    \caption{Ground-truth solution (top row) and our neural network solution (middle row) with deep minimizing movement scheme at time $t = 0.01, 0.1, 0.2$.  (At $t = 2\tau, 20\tau, 40\tau$) The relative $L^2$ error for each timestep is plotted in the graph at the bottom row.}
    \label{fig:heat_2d}
\end{figure}

\paragraph{8-dimensional case}
In the experiment for 8-dimensional heat equation, we set $\kappa=0.25$, and $a = \pi/4$. The domain $\Omega$ is given by $[-\pi/4, \pi/4]^{8}$. For an initial condition $u_0(x) = 1 - 0.2\cos{8x_2} + 2\cos{4x_4} - \cos{4x_6} + \cos{4x_8} + 2\cos{4x_1}\cos{4x_3}\cos{4x_8} + 2\cos{4x_2}\cos{4x_5}\cos{4x_7}$, its ground-truth solution $u(t,x)$ is computed by (\ref{heat_known_sol}), the same as the 2-dimensional case. A size of the time-step for the minimizing movement scheme is given by $\tau = 0.005$.

\paragraph{Model training details for 8d case}
For the neural network model that approximates the solution $u(x)$, a 3-layer fully connected neural network with 256 hidden units is used. A hyperbolic tangent (Tanh) activation function is applied. We choose the learning rate as $10^{-5}$ with an Adam optimizer. While training the neural network model, we uniformly picked 50,000 samples from $\Omega = [-\pi/4,\pi/4]^8$ and computed the minimization target (\ref{eq:L2_target}) by the Monte-Carlo approximation.

\begin{figure}[h]
    \centering
    \includegraphics[width=0.99\textwidth]{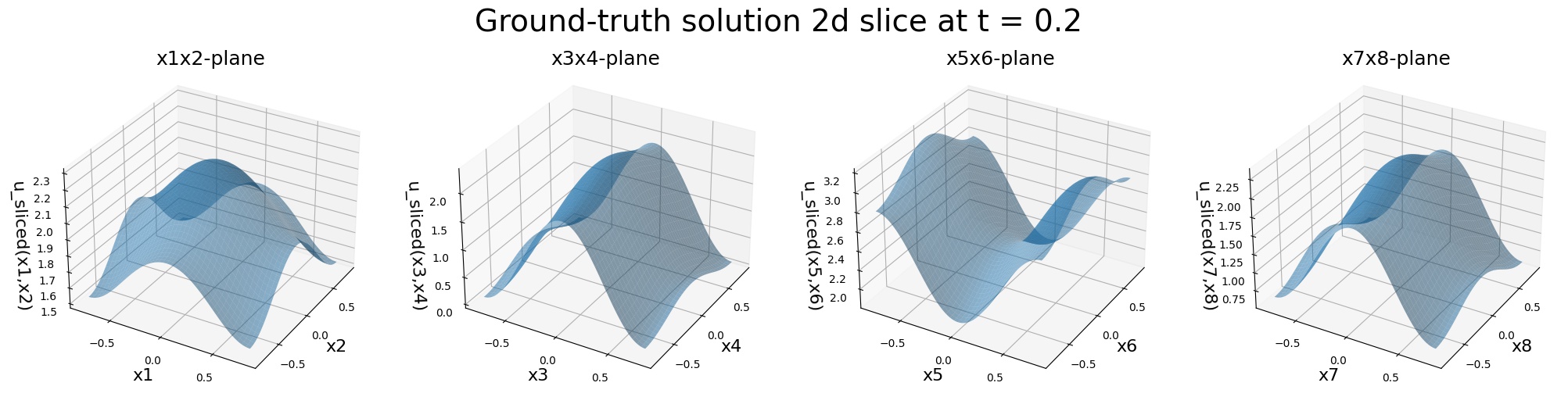}
    \includegraphics[width=0.99\textwidth]{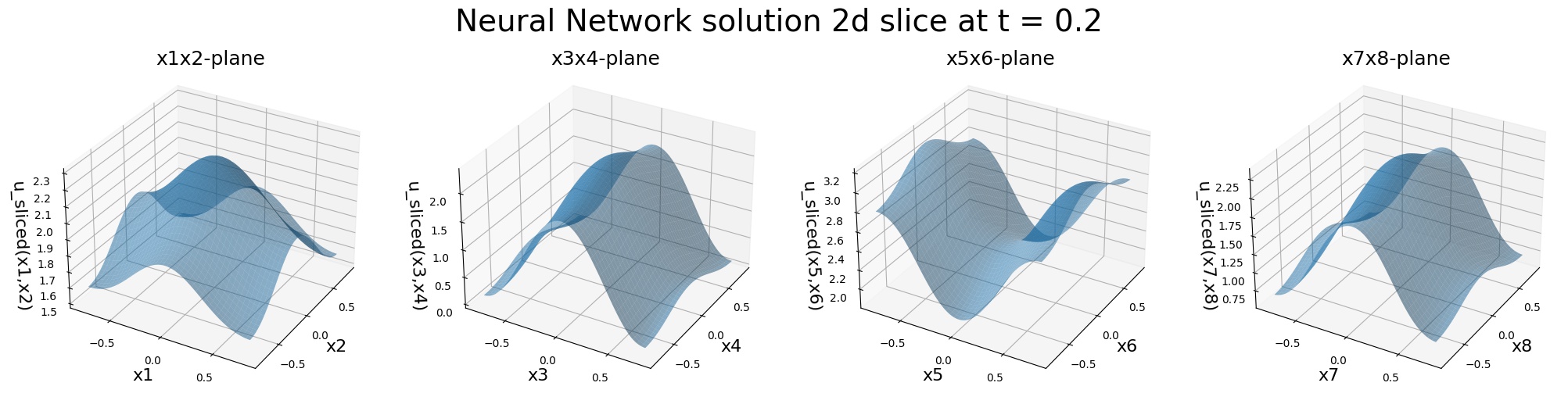} \\
    \vspace{0.4cm}
    \includegraphics[width=0.99\textwidth]{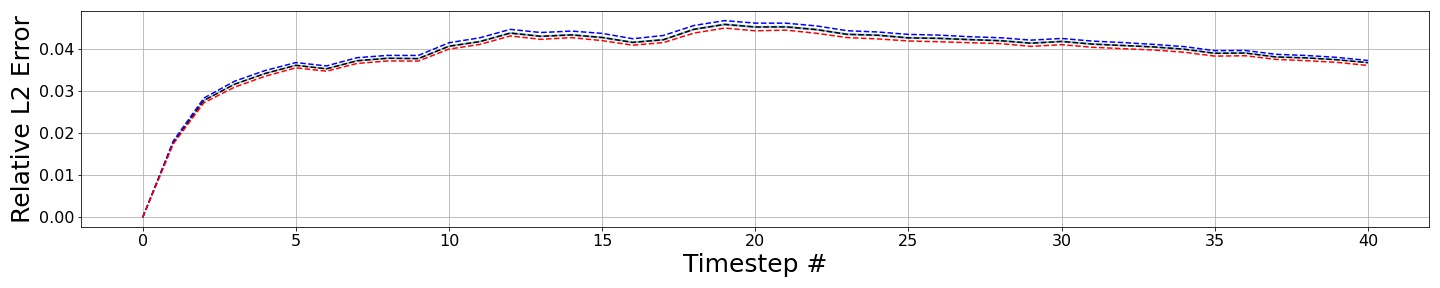}
    \caption{Ground-truth solution (top row) and our neural network solution (middle row) with deep minimizing movement scheme at time $t = 0.2 = 40\tau$. The relative $L^2$ error for each timestep is plotted in the graph at the bottom row. In the relative $L^2$ error plot, we estimated the error with 500 trials, each evaluated on 50,000 points that are uniformly sampled from $\Omega$. Blue dotted line indicates the maximum error and the red dotted line indicates the minimum error among the 500 trials for each timestep. As the plots of the maximum error and the minimum error are almost coincident, it can be deduced that the evaluated result is surely robust.}
    \label{fig:heat_8d}
\end{figure}
 
\paragraph{Results - 8d case}
For the high dimensional heat equation, we illustrated our approximated neural network solution and its relative $L^2$ error result in Figure \ref{fig:heat_8d}, by visualizing the solution on the 4 hyperplanes $x_1x_2$, $x_3x_4$, $x_5x_6$, $x_7x_8$-planes. Each solution $u_{sliced}(x_i,x_{i+1})$ was evaluated and visualized by fixing the other 6 variables to zero and varying $(x_i, x_{i+1}) \in [-\pi/4, \pi/4]^2$ for $i = 1, 3, 5, 7$. We plotted the ground-truth solution and the approximated neural network solution at the time $t=40\tau$. Its relative $L^2$ error is estimated 500 times, each with 50,000 points that are uniformly sampled from the domain $\Omega$. The deep minimizing movement scheme also achieved a good result on the 8d heat equation. Our approximated neural network solution achieved the relative $L^2$ error around $3.67\times10^{-2}$ at $t=0.2=40\tau$. These results clearly showed the applicability of our proposed method as a cornerstone of dealing with high dimensional PDEs. 

\subsubsection{Allen-Cahn equation: 2-dimensional}

The Allen-Cahn equation
\begin{equation}
    \begin{cases}
        \partial_t u = \Delta u - \epsilon^{-2}F'(u) \quad \text{in } \Omega,\\
        \nabla u \cdot \hat{\boldsymbol{n}}=0 \quad \text{on } \partial \Omega,
    \end{cases}
    \label{ac_eq}
\end{equation}

where $\boldsymbol{n}$ denotes the unit normal vector, is the $L^2$-gradient flow of the functional $I_\epsilon(u):=\frac{1}{2}\int ||\nabla u||^2 dx + \epsilon^{-2} \int F(u) dx$, where $\epsilon>0$ and $F$ is a double well potential. In our experiment, we consider the case where the double well potential $F(s)$ is given by $\frac{(s^2-1)^2}{4}$. We set $\epsilon=0.25$ and the domain $\Omega = [-2,2]\times[-2,2]$. The ground-truth solution of (\ref{ac_eq}) is computed with a forward Euler scheme, where the initial condition in Fig. \ref{aceq_init_figs} is examined. A size of the time-step for the minimizing movement scheme is given by $\tau = 0.005$.
\begin{figure}[h]
\centering
    \includegraphics[width=0.17\textwidth]{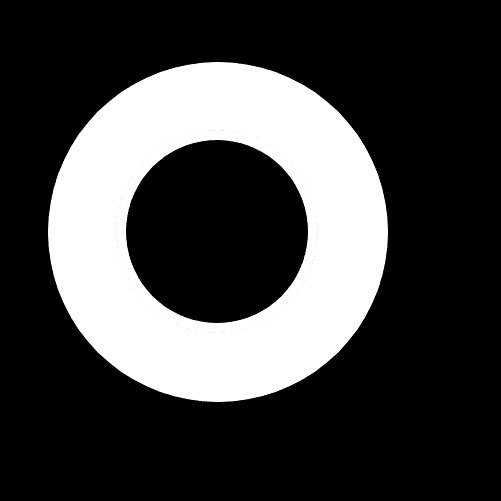}
    \includegraphics[width=0.17\textwidth]{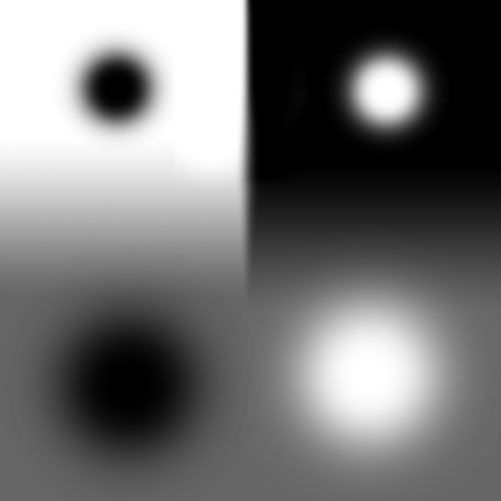}
    \caption{A hard torus-shape initial condition (left) and a 4-point initial condition (right) are used for the experiment on the Allen-Cahn equation. Black color indicates -1 and white color indicates 1.}
    \label{aceq_init_figs}
    \vspace{-0.5cm}
\end{figure}
\paragraph{Model training details} For the neural network model that approximates the solution $u(x)$, a 2-layer fully connected neural network with 256 hidden units is used. An exponential linear unit (ELU) activation function is applied. We choose the learning rate as $10^{-3}$ with an Adam optimizer. As in the 2d heat equation case, we uniformly picked 10,000 samples from $\Omega = [-2, 2]^2$ and computed the minimization target.
\paragraph{Results} The results for the hard torus-shape initial condition and the 4-point initial condition in Fig. \ref{aceq_init_figs} are visualized in Fig. \ref{aceq_result1} and Fig. \ref{aceq_result2} respectively. The ground-truth solution is obtained with the forward Euler scheme. We visualized the ground-truth solution and the approximated neural network solution at the time $t=2\tau, 20\tau, 40\tau$ and plotted the relative $L^2$ error for each timestep. The trained models for the two initial conditions at $t=0.2=40\tau$ achieved its relative $L^2$ error $3.76\times10^{-2}$ and $6.69\times10^{-2}$ respectively. Although the borderline area between -1 and +1 evolves rapidly and has a steep slope, the neural network approximator and our proposed deep minimizing movement scheme achieved good results for the given two initial conditions.

\begin{figure}[h]
    \centering
    \includegraphics[width=0.95\linewidth]{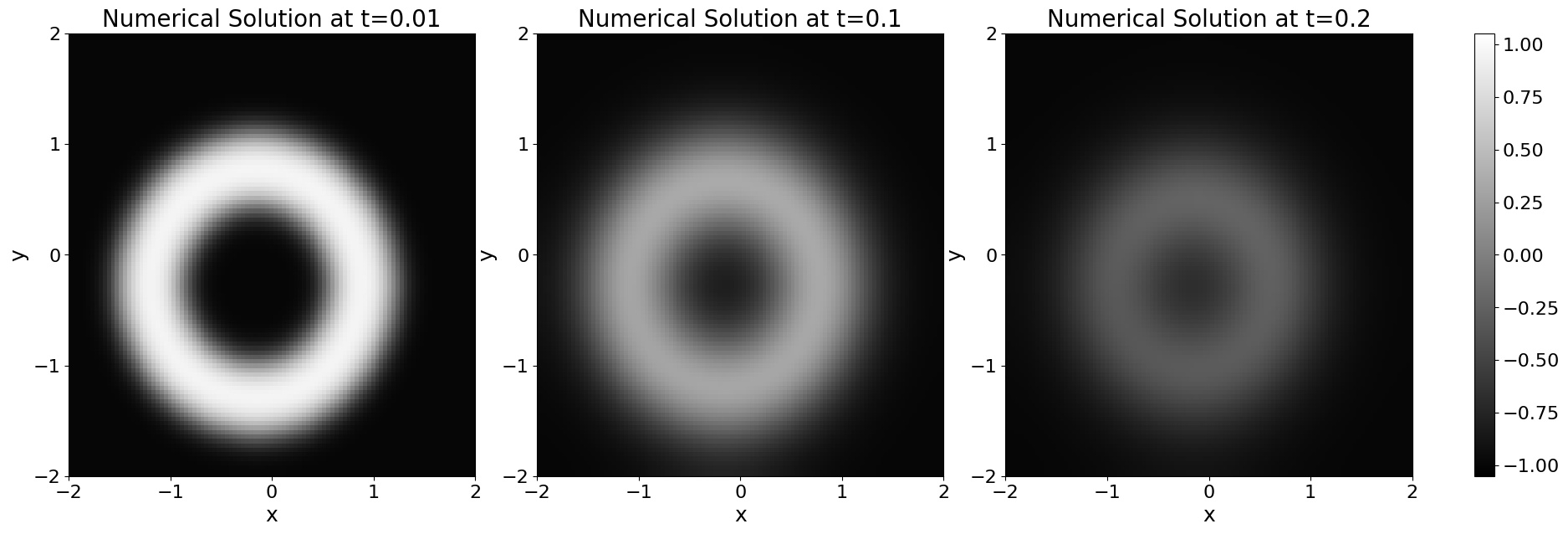}
    \includegraphics[width=0.95\linewidth]{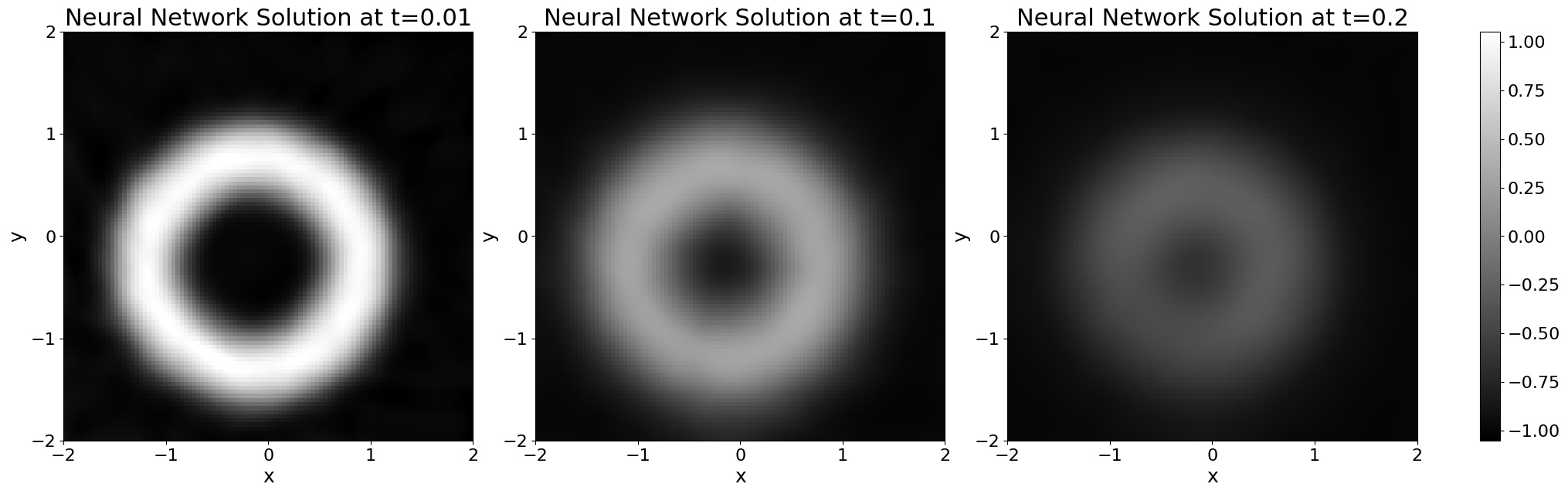} \\
    \vspace{0.1cm}
    \includegraphics[width=0.95\linewidth]{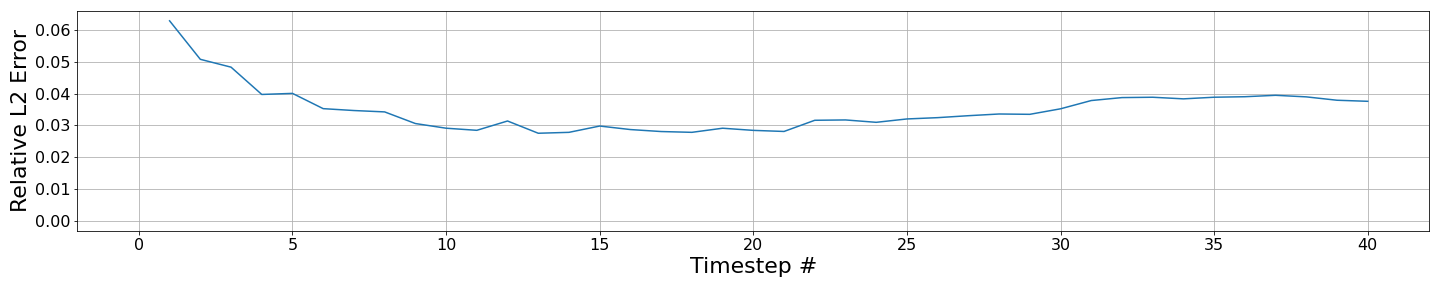}
    \caption{Ground-truth numerical solution (top row) and our neural network solution (middle row) with deep minimizing movement scheme of the Allen-Cahn equation at time $t = 0.01, 0.1, 0.2$. (At $t = 2\tau, 20\tau, 40\tau$), based on the hard torus-shape initial condition. Its relative $L^2$ error for each timestep is plotted in the graph at the bottom row.}
    \label{aceq_result1}
\end{figure}

\begin{figure}[h]
    \centering
    \includegraphics[width=0.95\linewidth]{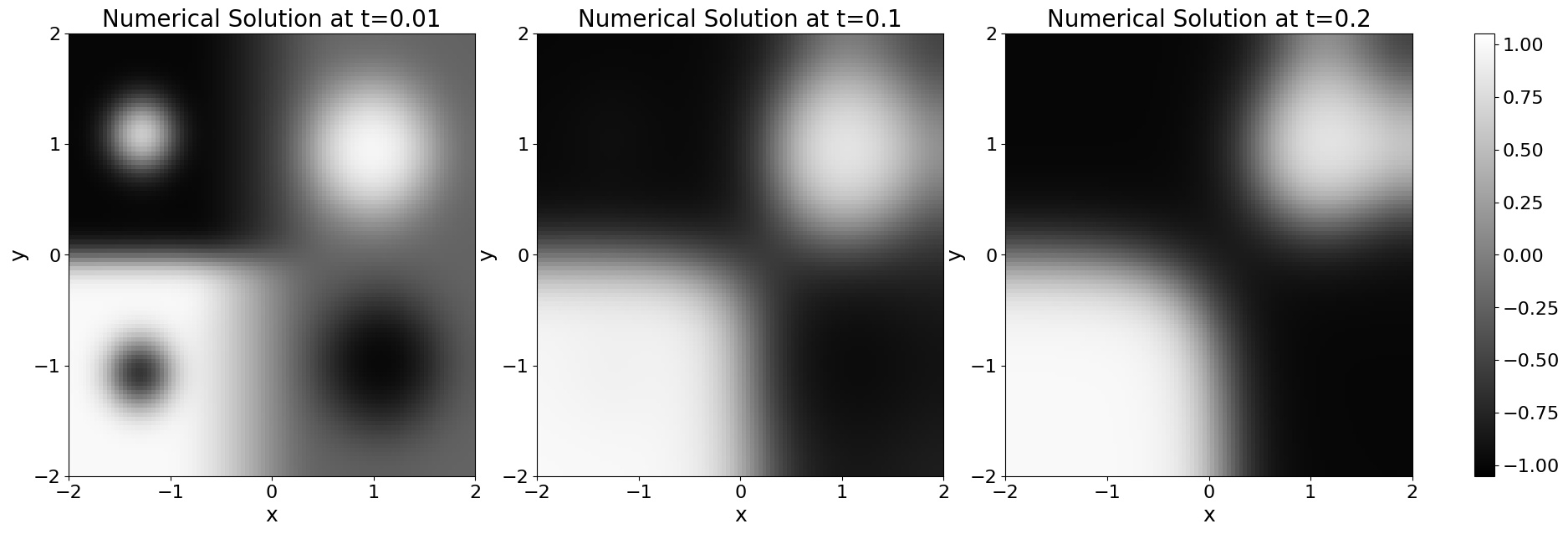}
    \includegraphics[width=0.95\linewidth]{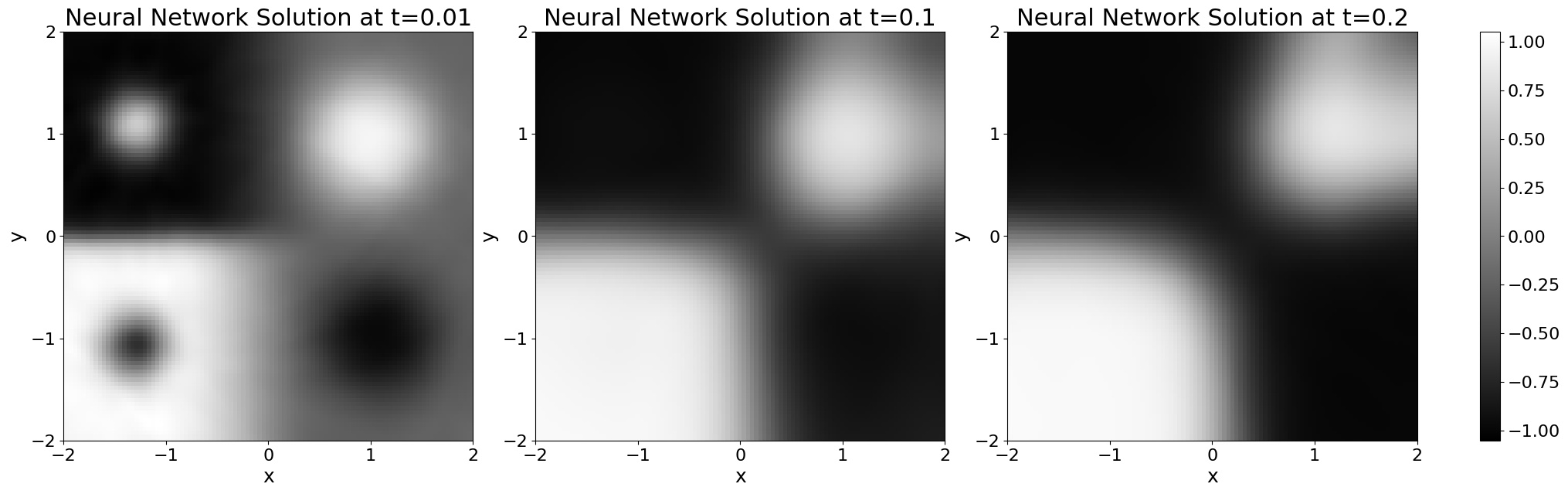} \\
    \vspace{0.1cm}
    \includegraphics[width=0.95\linewidth]{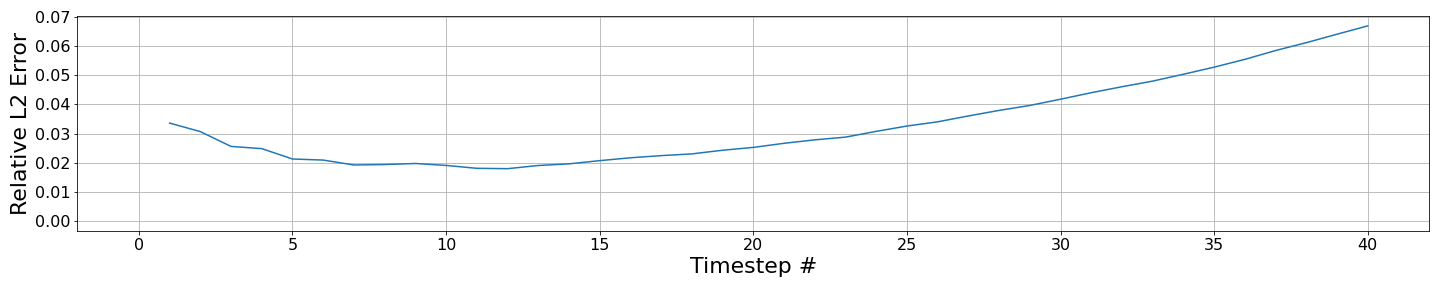}
    \caption{Ground-truth numerical solution (top row) and our neural network solution (middle row) with deep minimizing movement scheme of the Allen-Cahn equation at time $t = 0.01, 0.1, 0.2$. (At $t = 2\tau, 20\tau, 40\tau$), based on the 4-point initial condition. Its relative $L^2$ error for each timestep is plotted in the graph at the bottom row.}
    \label{aceq_result2}
\end{figure}

\subsection{\texorpdfstring{$\boldsymbol{\mathbb{W}_2}$}{TEXT}%
-gradient flows}

In this section, we consider several examples of $\mathbb{W}_2$-gradient flows. In the following experiments, we used the Softplus activation function for the last layer to guarantee the positivity of neural networks. Activation functions used in the remaining layers except for the last layer are specified for each experiment.

\subsubsection{Heat equation: 2-dimensional}
We consider 
\begin{equation}\label{heat_eqn_W2}
    \begin{cases}
        \partial_t u = \kappa\Delta u \quad \text{in } [0,T] \times [-\pi, \pi]^2, \\
        \nabla u \cdot \hat{\boldsymbol{n}}=0 \quad \text{on } [0,T] \times \partial [-\pi, \pi]^2, \\
        u(0,x_1, x_2) =  (1 + \frac{1}{5} \cos(x_1) + \frac{1}{5} \cos(x_2)) / (4 \pi^2),
    \end{cases} \nonumber
\end{equation} which corresponds to a $\mathbb{W}_2$-gradient flow of the entropy functional $\mathcal{F}(u):=\int_\Omega u(x) \log u(x) dx$.

\paragraph{Model training details} In this numerical experiment, we employ a 4-layer fully connected neural network with 512 hidden units with cosine activation function. We use Adam optimizer with a learning rate $2\times 10^{-4}$. We set $\tau = 0.005$. For each training epoch, we sample 20,000 points from $[-\pi,\pi]\times[-\pi,\pi]$ uniformly. Numerical integrations are computed by using the Monte-Carlo method.

\begin{figure}[h]
     \centering
     \begin{subfigure}[b]{0.5\textwidth}
         \centering
         \includegraphics[width=\textwidth]{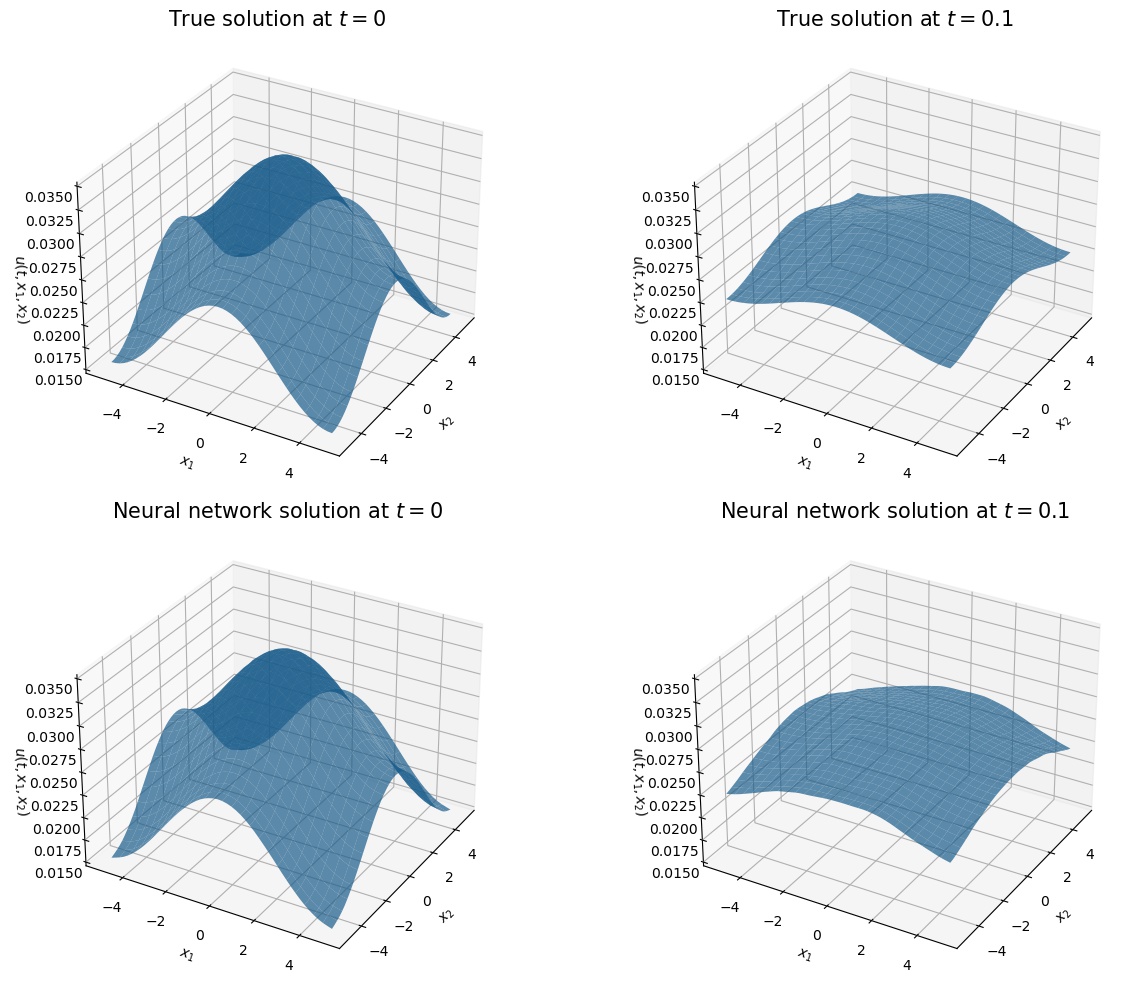}
         \label{fig:Heat_2d_solution}
     \end{subfigure}
     \hfill
     \begin{subfigure}[b]{0.48\textwidth}
         \centering
         \includegraphics[width=\textwidth]{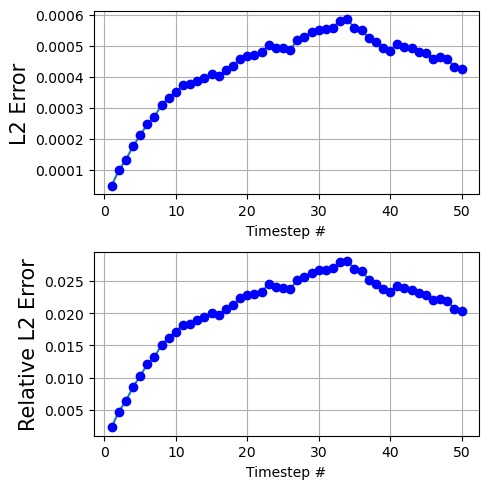}
         \label{fig:Heat_2d_error}
     \end{subfigure}
     \hfill
        \caption{Left: Analytic solutions and approximated solutions of \eqref{heat_eqn_W2} at $t=\tau, 51\tau$. Right: $L^2$ error and relative $L^2$ error of an approximated solution.}
        \label{fig:Heat_2d}
\end{figure}

\subsubsection{Porous Medium: 1-dimensional}

The porous medium equation
\begin{equation}\label{PM_eqn}
    \partial_t u(t,x) = \Delta u^m, \quad \text{in } [0,T]\times\mathbb{R}^n
\end{equation}
is the $\mathbb{W}_2$-gradient flow of the energy $\mathcal{F}(u):=\frac{1}{m-1}\int_{\mathbb{R}^n}u^m(x)dx$.

A fundamental example of exact solution of this equation was obtained independently by Barenblatt and Pattle \cite{carrillo2000asymptotic}, which are densities of the form

\begin{equation}\label{porous_sol}
    u(t,x) = (t+t_0)^{-kn}\bigg(C- \frac{(m-1)k}{2m}||x||^2t^{-2k}\bigg)_+^{\frac{1}{m-1}},
\end{equation}
where $k=(n(m-1)+2)^{-1}$ and the positive constant $C$ is defined by the identity $\int_{\mathbb{R}^n} u(t,x) dx = 1$.

In this experiment, we set $t_0=10^{-1}$, $m=2$, and $C=(\frac{3}{64})^{\frac{1}{3}}$. Moreover, we set $u(x,0)$ in the equation \eqref{porous_sol} to be the initial condition.

\paragraph{Model training details} In this numerical experiment, we employ a 5-layer fully connected neural network with 512 hidden units with ELU activation function. We use Adam optimizer with a learning rate $2\times 10^{-4}$. We set $\tau=0.005$ and truncated the domain $\mathbb{R}$ by $[-3,3]$ so that the numerical integration of the initial condition has a negligible error. For each training epoch, we sample $3,000$ points from $[-3,3]$ uniformly. Numerical integrations in the algorithm are computed by using the Monte-Carlo method.

\paragraph{Results}

Both the neural network solution and the ground-truth solution over time until $t=60\tau$ are plotted in Figure \ref{fig:PM_1d_sol} and error estimation is given in Figure \ref{fig:PM_1d_error}. As shown in Figure \ref{fig:PM_1d_sol}, despite the accumulation of errors over time, it fits well for a fairly long time.

\begin{figure}[h]
    \centering
    \includegraphics[width=\textwidth]{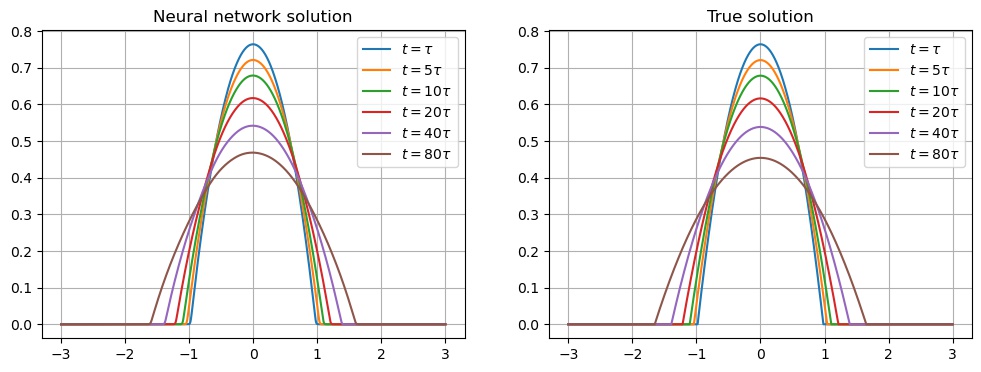}
\caption{Neural network solution and the ground-truth solution for the 1-dimensional porous medium equation over timesteps $t=\tau,\cdots, 60\tau$.}
\label{fig:PM_1d_sol}
\end{figure}

\begin{figure}[h]
    \centering
    \includegraphics[width=\textwidth]{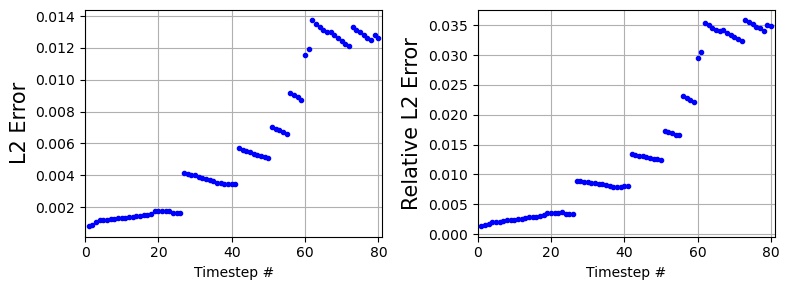}
\caption{Error estimation of the neural network solution for the porous medium equation.}
\label{fig:PM_1d_error}
\end{figure}

\subsubsection{Porous medium equation: 2-dimensional}
In this experiment, we set $t_0=10^{-1}, m=2$, and $C=\frac{1}{\sqrt{8\pi}}$. As in the $1$-dimensional case, we take the $u(x,0)$ in the equation $\eqref{porous_sol}$ to be the initial condition.

\paragraph{Model training details} In this numerical experiment, we employ a 5-layer fully connected neural network with 512 hidden units with ELU activation function. We use Adam optimizer with a learning rate $2\times 10^{-4}$. We set $\tau=0.005$ and truncated the domain $\mathbb{R}^2$ by $[-3,3]\times [-3,3]$ so that the numerical integration of the initial condition has a negligible error. For each training epoch, we sample $10,000$ points from $[-3,3]\times [-3,3]$ uniformly. Numerical integrations in the algorithm are computed by using the Monte-Carlo method.

\paragraph{Results}

We summarize the results in Figure \ref{fig:PM_2d} and Figure \ref{fig:PM_2d_error}. In Figure \ref{fig:PM_2d}, the neural network solution and the ground-truth solution are both plotted for timesteps $t=\tau$ and $t=60\tau$. As in the 1-dimensional case, neural network solution fits well for a fairly long time. It worths noticing that since we sampled more points in this experiment than in the 1-dimensional experiment, the errors are smaller.

\begin{figure}[h]
     \centering
     \begin{subfigure}[b]{0.9\textwidth}
         \centering
         \includegraphics[width=\textwidth]{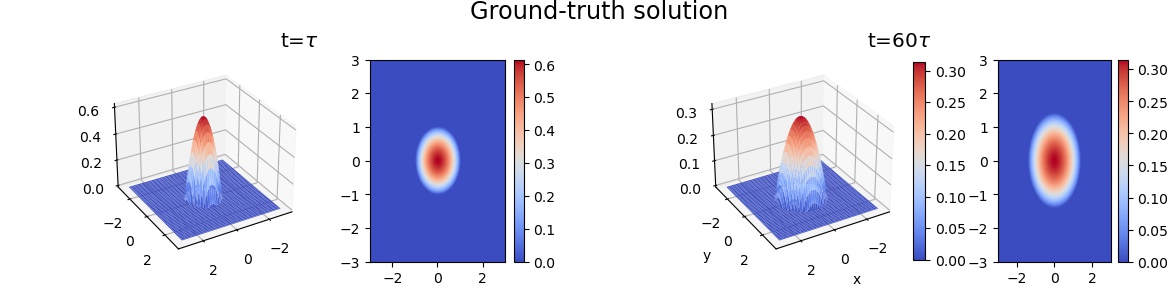}
         \label{fig:PM_2d_GT}
     \end{subfigure}
     \hfill
     \begin{subfigure}[b]{0.9\textwidth}
         \centering
         \includegraphics[width=\textwidth]{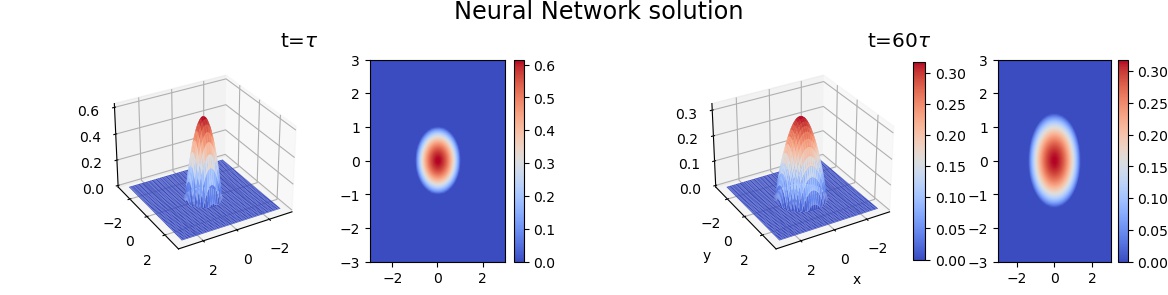}
         \label{fig:PM_2d_NN}
     \end{subfigure}
     \hfill
        \caption{Neural network solution and the ground-truth solution for the 2-dimensional porous medium equation for timesteps $t=\tau$ and $t=60\tau$.}
        \label{fig:PM_2d}
\end{figure}

\begin{figure}[h]
    \centering
    \includegraphics[width=\textwidth]{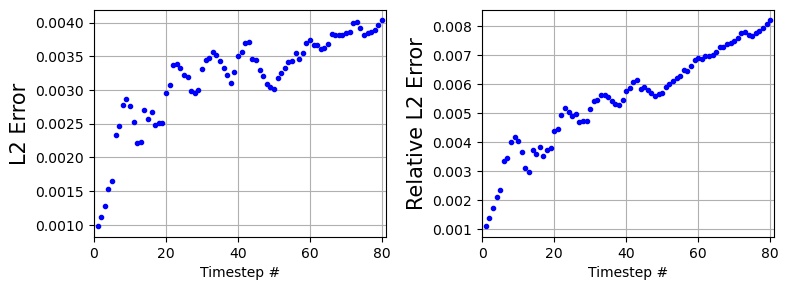}
\caption{Error estimation of the neural network solution for the 2-dimensional porous medium equation.}
\label{fig:PM_2d_error}
\end{figure}

\subsubsection{Fokker-Planck equation: 2-dimensional}
The Fokker-Planck equation
\begin{equation}\label{FP_eqn}
    \partial_t u(t,x) = \nabla \cdot (u \nabla V) + \Delta u, \quad \text{in } [0,T]\times\mathbb{R}^n
\end{equation}
is the $\mathbb{W}_2$-gradient flow of the energy $\mathcal{F}(u):=\int_{\mathbb{R}^2} V(x)u(x) + u(x)\log u(x)dx$. In this numerical example, we set the initial condition to be the standard Gaussian distribution and set $V(x) = \frac{1}{2} (x-\mu)^T \Sigma^{-1} (x-\mu)$, with $\mu = [\frac{1}{3}, \frac{1}{3}]$, and $\Sigma=\begin{bmatrix}
\frac{5}{8} & -\frac{3}{8} \\
-\frac{3}{8} & \frac{5}{8} 
\end{bmatrix}$.
Then, an analytic solution of \eqref{FP_eqn} reads
\begin{equation}\label{FP_2d_analytic}
    u(t,x) = \mathcal{N}(\mu(t), \Sigma(t)) := \frac{1}{2\pi \sqrt{|\Sigma(t)|}} \exp{-\frac{1}{2} (x-\mu(t))^T \Sigma(t)^{-1} (x-\mu(t))},
\end{equation}
where $\mu(t) = (1-e^{-4t})\mu$, $\Sigma(t) = \begin{bmatrix}
\frac{5}{8} + \frac{3}{8}\times e^{-8t} & -\frac{3}{8} + \frac{3}{8} \times e^{-8t} \\
-\frac{3}{8} + \frac{3}{8} \times e^{-8t} & \frac{5}{8} + \frac{3}{8} \times e^{-8t} 
\end{bmatrix}$.

\paragraph{Model training details} In this numerical experiment, we employ a 4-layer fully connected neural network with 512 hidden units with ReLU activation function. We use Adam optimizer with a learning rate $2\times 10^{-4}$. We set $\tau = 0.005$ and truncated the domain $\mathbb{R}^2$ by $[-5,5]\times[-5,5]$ so that the numerical integration of the initial condition has a negligible error. For each training epoch, we sample 10,000 points from $[-5,5]\times[-5,5]$ uniformly. Numerical integrations in the algorithm are computed by using the Monte-Carlo method.

\paragraph{Results} We summarize the results in Figure \ref{fig:FP_2d}. The left panel of Figure \ref{fig:FP_2d} shows both an analytic solution \eqref{FP_2d_analytic} and a neural network solution at $t=\tau$, and $t=21\tau$. The right panel of the figure shows trajectories of mean vector of an analytic solution and a neural network solution. We achieved a relative error less than 0.01 (averaged in time) in the truncated domain, where $u_{NN}$ denotes a neural network solution.

\begin{figure}[h]
     \centering
     \begin{subfigure}[b]{0.5\textwidth}
         \centering
         \includegraphics[width=\textwidth]{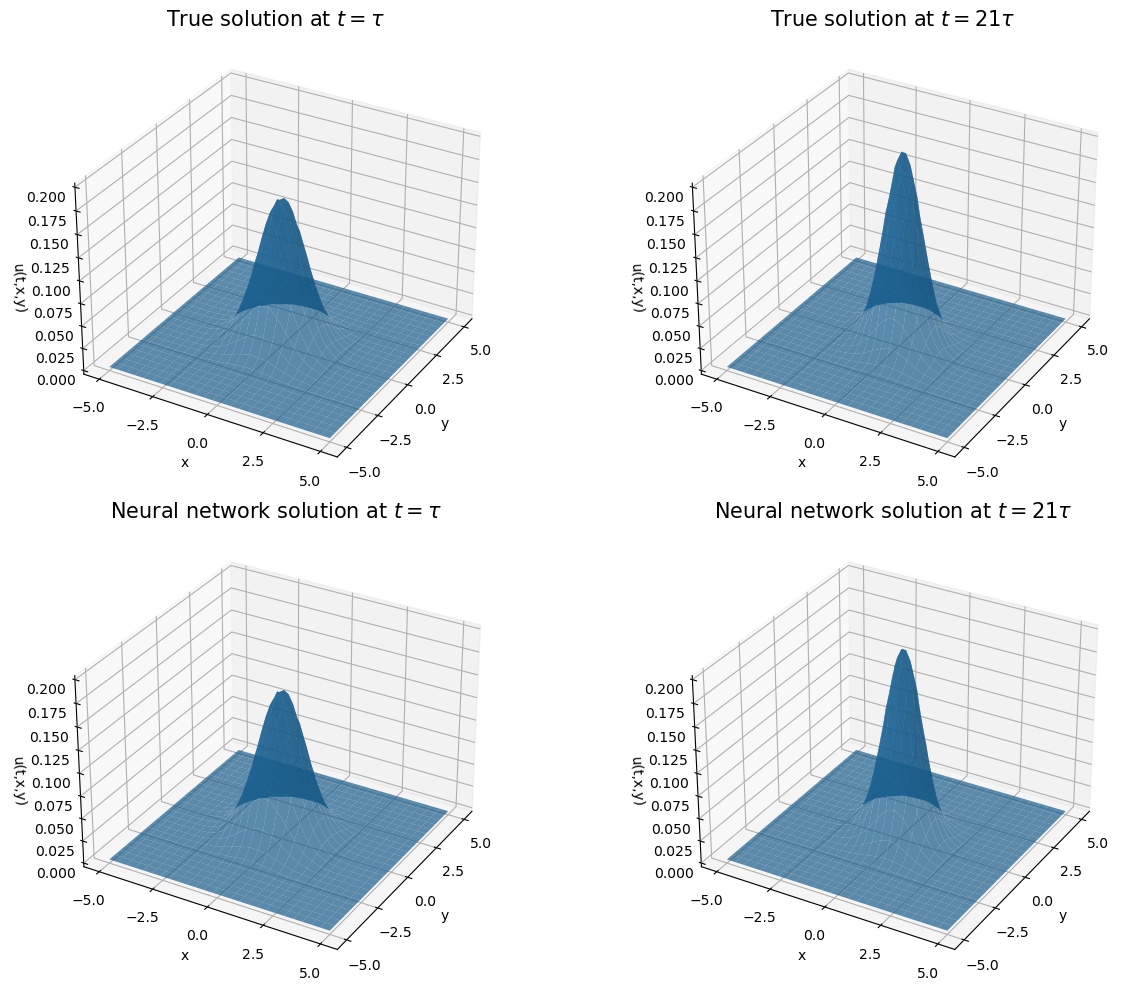}
         \label{fig:FP_2d_solution}
     \end{subfigure}
     \hfill
     \begin{subfigure}[b]{0.45\textwidth}
         \centering
         \includegraphics[width=\textwidth]{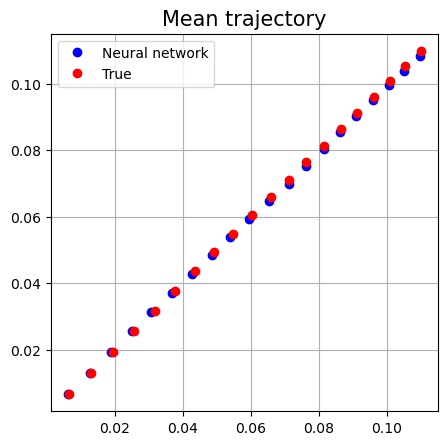}
         \label{fig:FP_2d_trajectory}
     \end{subfigure}
     \hfill
        \caption{Left: Analytic solutions and approximated solutions of \eqref{FP_eqn} at $t=\tau,\cdots, 21\tau$. Right: Mean trajectories of an analytic solution and an approximated solution.}
        \label{fig:FP_2d}
\end{figure}

\subsubsection{Fokker-Planck equation: 4-dimensional}

In this experiment, we set the initial condition to be the standard Gaussian distribution and $V(x)=\frac{1}{2}(x-\mu)^T \Sigma^{-1} (x-\mu)$, where $\mu=[\frac{1}{3},\frac{1}{3},0,0]$ and $\Sigma = \begin{bmatrix}
\frac{5}{8} & -\frac{3}{8} \\
-\frac{3}{8} & \frac{5}{8} 
\end{bmatrix} \bigoplus \begin{bmatrix}
1 & 0 \\
0 & 1 
\end{bmatrix}$. In this case, the analytic solution of \eqref{FP_eqn} is given by \eqref{FP_2d_analytic}, where $\mu(t)=[\frac{1}{3}(1-e^{-4t}), \frac{1}{3}(1-e^{-4t}), 0, 0]$ and $\Sigma(t) = \begin{bmatrix}
\frac{5}{8} + \frac{3}{8}\times e^{-8t} & -\frac{3}{8} + \frac{3}{8} \times e^{-8t} \\
-\frac{3}{8} + \frac{3}{8} \times e^{-8t} & \frac{5}{8} + \frac{3}{8} \times e^{-8t} 
\end{bmatrix} \bigoplus \begin{bmatrix}
1 & 0 \\
0 & 1 \end{bmatrix}$.

\paragraph{Model training details} In this numerical experiment, we employ a 5-layer fully connected neural network with 256 hidden units with ELU activation function. We use Adam optimizer with a learning rate $2\times 10^{-4}$. We set $\tau=0.005$ and truncated the domain $\mathbb{R}^4$ by $[-3.5,3.5]^4$, and in this case, the error of the numerical integration of the initial condition on this domain and the whole domain is less than $2\%$. For each training epoch, we sampled 160,000 points from $[-3.5,3.5]^4$ uniformly. Numerical integrations in the algorithm are computed by using the Monte-Carlo method.

\paragraph{Results} The mean trajectories of the neural network solution and the ground-truth solution along timesteps $t=\tau,\cdots,60\tau$ are given in Figure \ref{fig:FP_4d_trajectory}. The first and second coordinates of the mean trajectories of both neural network solution and the ground-truth solution are in the relation $y=x$ as illustrated in the left top of Figure \ref{fig:FP_4d_trajectory}. The right panel of Figure \ref{fig:FP_4d_trajectory} is the mean trajectories of both neural network solution and the ground-truth solution over $t=0,\cdots, 20\tau$, and it can be checked that they are overlapping.

\begin{figure}[h]
    \centering
    \includegraphics[width=\textwidth]{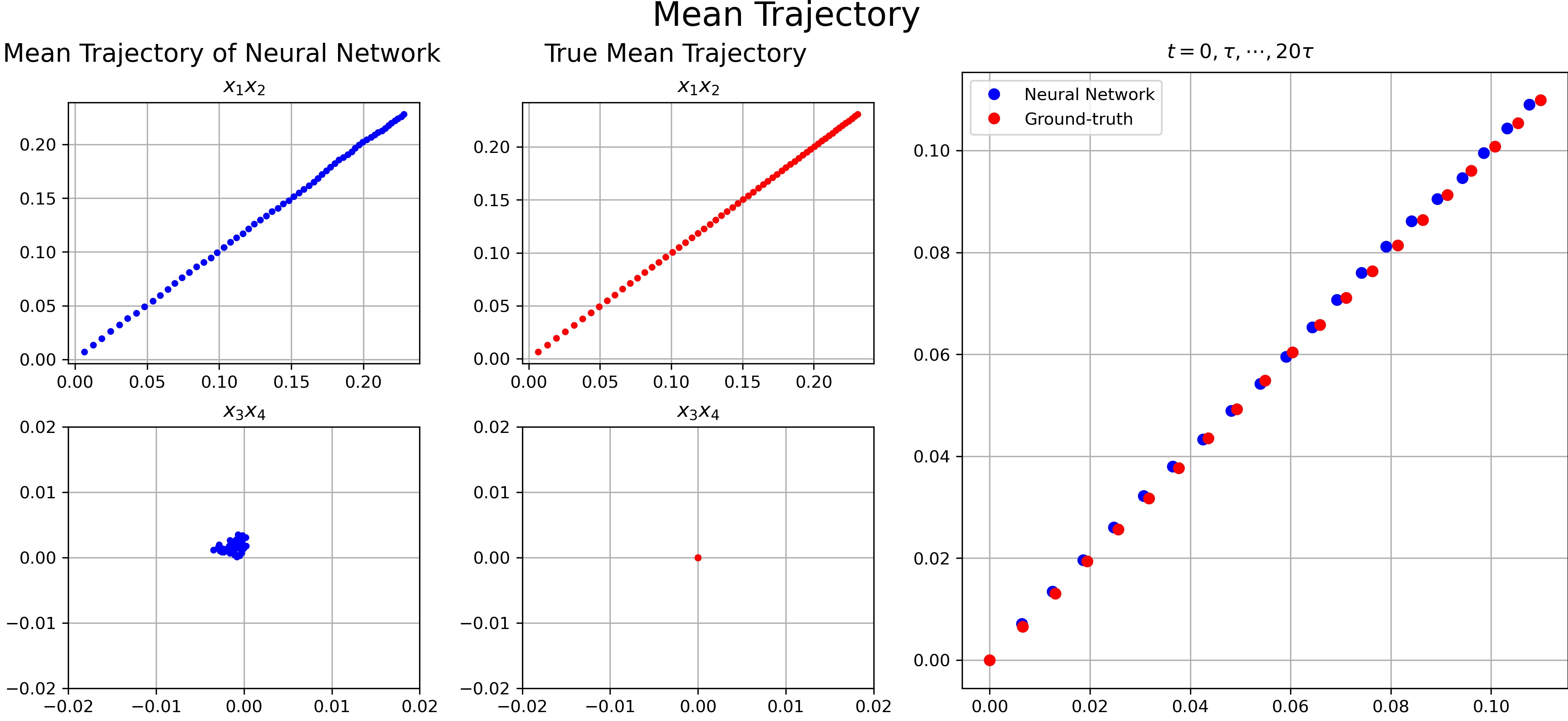}
\caption{Mean trajectories of our optimized neural network and the ground-truth solution. The trajectories taken from $\tau$ to $60\tau$ are plotted with respect to the first and the second coordinates on the top left, and with respect to the third and fourth coordinates on the bottom left. On the right, trajectories of neural network and the ground-truth solution taken for $t=0$~$20\tau$ are plotted in the same plot together for detailed comparison.}
\label{fig:FP_4d_trajectory}
\end{figure}

The heat maps of cross-sectional $u_{sliced}$'s are given in Figure \ref{fig:FP_4d}. The ground-truth solution of $u_{sliced}(x_3,x_4)$ converges to a Gaussian distribution with the variance $\begin{bmatrix}
\frac{5}{8} & -\frac{3}{8} \\
-\frac{3}{8} & \frac{5}{8} 
\end{bmatrix}$ while the ground-truth solution of $u_{sliced}(x_1,x_2)$ remains unchanged. As time goes by, the shape of the variance of the neural network solution of $u_{sliced}(x_3,x_4)$ follows $\begin{bmatrix}
\frac{5}{8} & -\frac{3}{8} \\
-\frac{3}{8} & \frac{5}{8} 
\end{bmatrix}$. Meanwhile, even with the passage of time, the neural network solution of $u_{sliced}(x_3,x_4)$ remains the same as the standard Gaussian distribution.

\begin{figure}[h]
     \centering
     \begin{subfigure}[b]{0.95\textwidth}
         \centering
         \includegraphics[width=\textwidth]{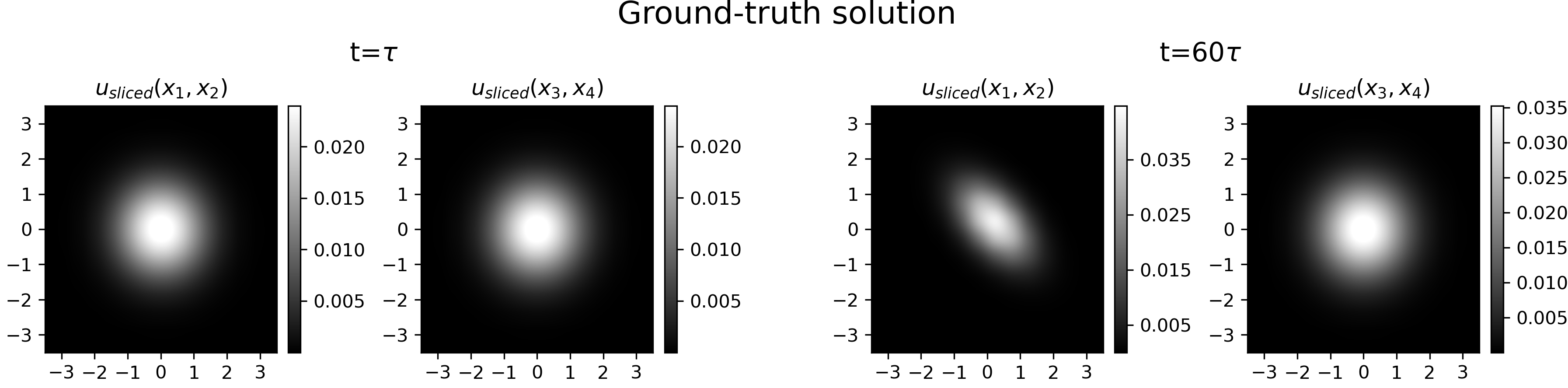}
         \label{fig:FP_4d_GT_slice}
     \end{subfigure}
     \hfill
     \begin{subfigure}[b]{0.95\textwidth}
         \centering
         \includegraphics[width=\textwidth]{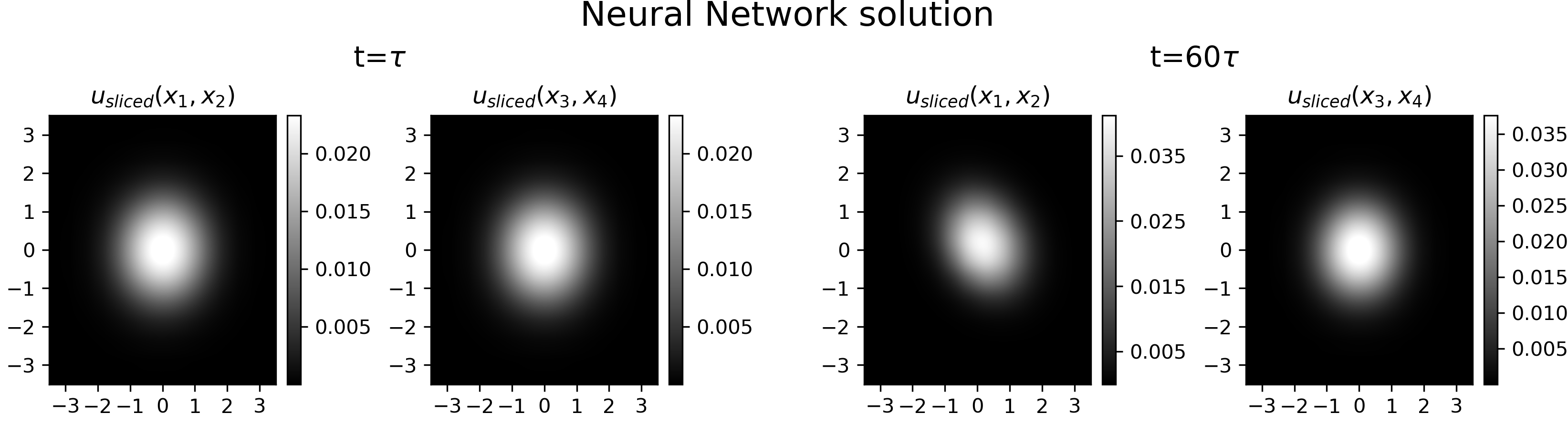}
         \label{fig:FP_4d_NN_slice}
     \end{subfigure}
     \hfill
        \caption{Heat maps of cross-sectional $u_{sliced}$ for the neural network solution and the ground-truth solution at the origin for timesteps $t=\tau$ and $t=60\tau$ respectively. Our neural network solution for $u_{sliced}(x_1,x_2)$ becomes anisotropic, leaving $u_{sliced}(x_3,u_4)$ isotropic, as time passes. Since the ground-truth solution has this property, this pattern is a desiderata for approximated solutions.}
        \label{fig:FP_4d}
\end{figure}

Error estimation is given in Figure \ref{fig:FP_4d_error}. The right panel of Figure \ref{fig:FP_4d_error} is the plot of the mean of the relative $L^2$ error over time, i.e. $\frac{1}{K}\sum_{k=1}^K [\text{relative }L^2 \text{ error at }t=k\tau]$.

\begin{figure}[h]
    \centering
    \begin{subfigure}[b]{0.45\textwidth}
    \centering
        \includegraphics[width=\linewidth]{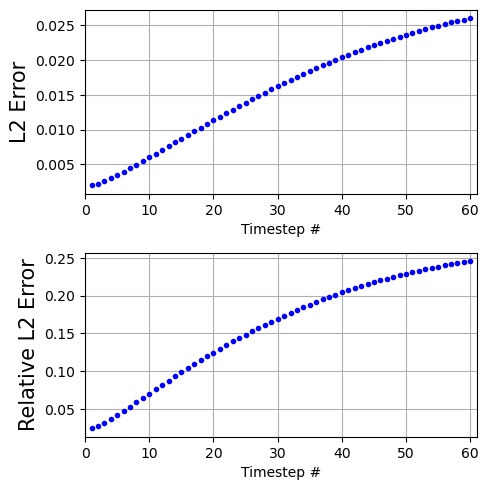}
    \end{subfigure}
    \begin{subfigure}[b]{0.45\textwidth}
    \centering
        \includegraphics[width=\linewidth]{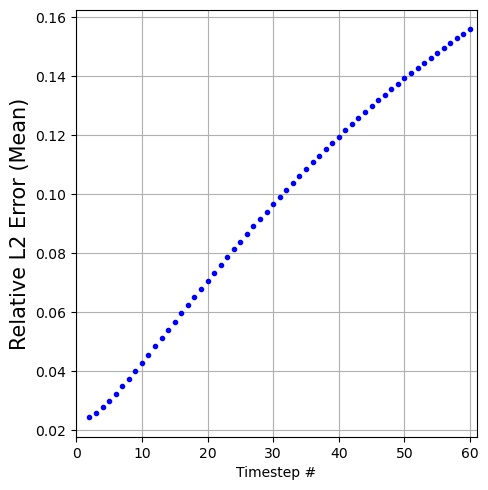}
    \end{subfigure}
\caption{Errors of the neural network solution for $4$-dimensional Fokker-Planck equation.}
\label{fig:FP_4d_error}
\end{figure}

\subsubsection{Effect of the convex regularization parameter \texorpdfstring{$\lambda$}{TEXT}}
We construct an additional experiment on the effect of the regularization parameter $\lambda$ in (\ref{W2objective}) for the 2-dimensional heat equation in Section \ref{subsubsec:2dheat}, where smaller or larger $\lambda$ leads to different results. This parameter approximates the convex-conjugate constraint by softly regularizing $R(\theta, \omega)$ term in (\ref{W2objective}). Here, we consider the same 2D Heat equation in Section \ref{subsubsec:2dheat}.

\paragraph{Model training details} We only changed $\lambda = 0.01, 0.1, 0.2, 0.5, 1, 2, 5, 10, 20, 50$, while keeping all other training parameters such as model size, epochs, and learning rate to the same as Section \ref{subsubsec:2dheat}. 

\begin{figure}[h]
    \centering
    \begin{subfigure}[b]{0.95\textwidth}
    \centering
        \includegraphics[width=\linewidth]{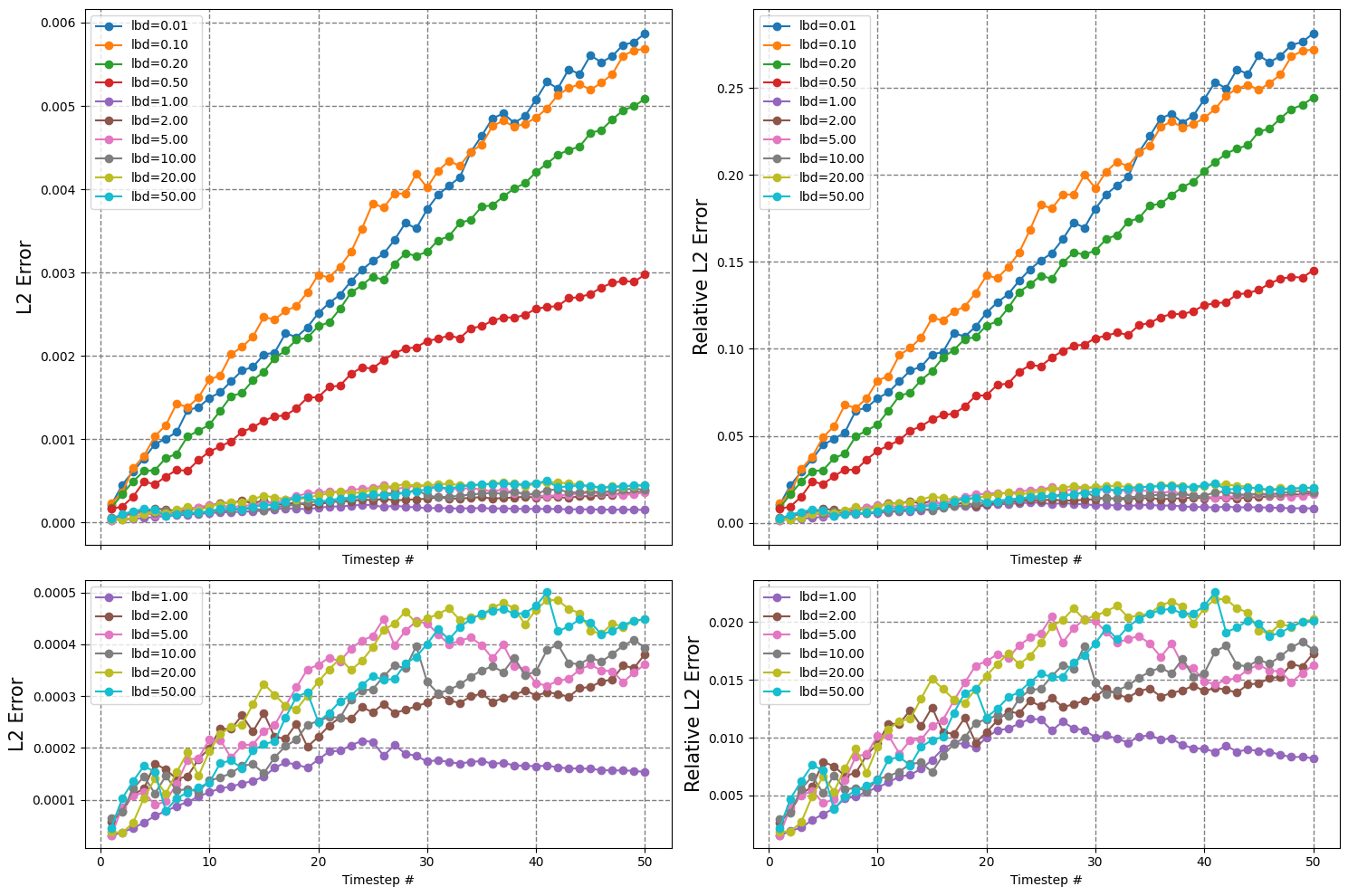}
    \end{subfigure}
\caption{Errors of the neural network solutions with different lambda values for a 2-dimensional heat equation. The results for $\lambda=0.01, 0.1, 0.2, 0.5, 1, 2, 5, 10, 20, 50$ are compared and plotted with different colors. The upper boxes show all cases, whereas the lower boxes show the cases $\lambda \geq 1$ for a detailed comparison.}
\label{fig:lambda_exp_error}
\end{figure}

\begin{figure}[h]
    \centering
    \begin{subfigure}[b]{0.95\textwidth}
    \centering
        \includegraphics[width=\linewidth]{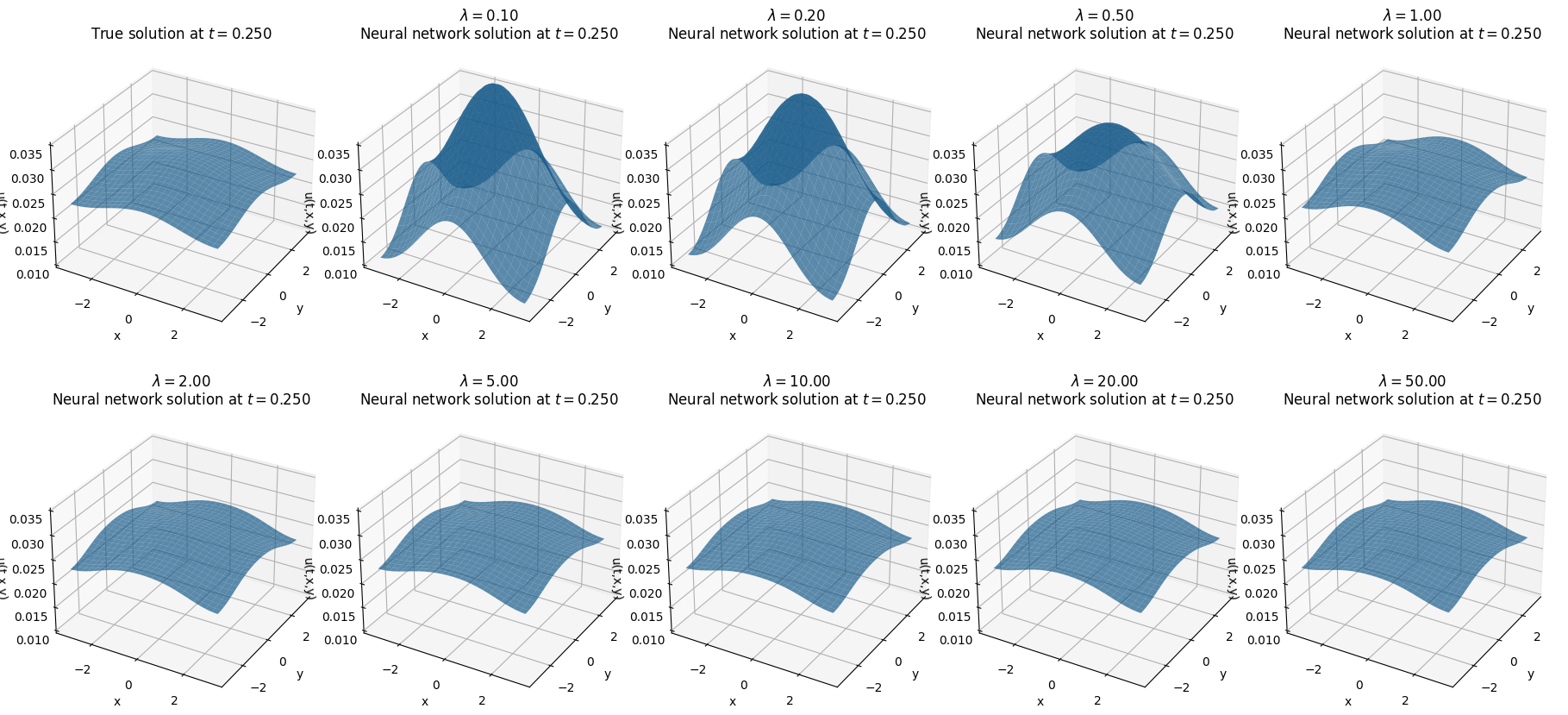}
    \end{subfigure}
\caption{A true solution and the neural network solutions with nine different $\lambda$ values for a 2-dimensional heat equation. The true solution and the results for $\lambda=0.1, 0.2, 0.5, 1, 2, 5, 10, 20, 50$ are visualized and plotted. The $\lambda=0.01$ case was omitted, as it can be easily checked to be failed by its high L2 error and relative L2 error.}
\label{fig:lambda_exp_solution}
\end{figure}

\paragraph{Results} The results for $\lambda = 1, 2, 5, 10, 20, 50$ achieved similar performance to the original result in Section \ref{subsubsec:2dheat}. Among the good results, setting $\lambda = 1$ shows the best performance. The relative error gets larger as the value of $\lambda$ increases from 1 to 50, but they were negligible compared to the error that occurred when the $\lambda$ value was small. The results for $\lambda = 0.01, 0.1, 0.2, 0.5$ failed to estimate $L^2$-Wasserstein distance between two distributions, which leads to a failed optimization process. This clearly implies that a small $\lambda$ value cannot regularize the convex-conjugate condition and leads to undesired optimization results. Detailed error plots can be found in Fig. \ref{fig:lambda_exp_error} and solutions for $t=0.25$ can be found in Fig. \ref{fig:lambda_exp_solution}.

Although the necessary value for $\lambda$ depends on the dimension of the $x$, in our $\mathbb{W}_2$-gradient flow experiments above we empirically observe that setting $\lambda=1$ for estimating $L^2$-Wassertein distance was enough. 

If the equation becomes more complex or the dimension of $x$ becomes higher, the value required for convex-conjugate regularization will become larger. This can be also empirically deduced from the previous study \cite{korotin2019wasserstein}, which chooses higher $\lambda$ as the task get more difficult. For example, the previous study set $\lambda$ to 1 for a 2-dimensional toy example, $\min(50, D)$ for a $D$-dimensional Gaussian optimal transport task, and 35,000 for an image-to-image style transfer task.

\section{Discussion}
In the experiments on several gradient flow equations above, the proposed mesh-free deep minimizing movement scheme achieved outstanding results for both low and high dimensions. Despite lack of mature error analysis, it has been known that neural networks are more fluent than the classical numerical methods because they are free from mesh-generation. By the series of experiments in this work, we verified that neural networks perform great for both low and high dimensions, and this indicates the possibility of applying our method to solving various PDEs related to gradient flows.

As pointed out in \cite{berg2018unified}, mesh generation often fails to scale when the domain has a complex geometry. Our approach for both $L^2$ and $\mathbb{W}_2$ cases, which is completely free of mesh generation, is therefore directly applicable for the gradient-flow type equation in a domain with complex geometry. Our method also takes advantage in that it can be directly applied not only to $L^2$ and $\mathbb{W}_2$ spaces, but also to any general space. Therefore, if some classes of PDEs are realized as gradient flows in some spaces other than $L^2$ and $\mathbb{W}_2$ in future, our method can be a breakthrough for numerically solving those equations.




\section*{Acknowledgement}
H. J. Hwang is supported by the National Research Foundation of Korea (NRF) grant funded by the Korea government (MSIT) (No. NRF-2017R1E1A1A03070105 and NRF-2019R1A5A1028324), Institute for Information \& Communications Technology 
Promotion (IITP) grant funded by the Korea government (MSIT) (No. 2019-0-01906,
Artificial Intelligence Graduate School Program (POSTECH)), and the Information 
Technology Research Center (ITRC) support program (No. IITP-2018-0-01441). H. Son is supported by National Research Foundation of Korea (NRF) grants funded by the Korean government (MSIT) (No. NRF-2019R1A5A10\allowbreak 28324).

\bibliographystyle{siamplain}
\bibliography{natbib}

\end{document}